\documentclass[11pt,a4paper]{article}
\usepackage{amssymb}
\usepackage{amsmath}
\usepackage{latexsym}
\usepackage{amsthm}
\usepackage{dsfont}

\newtheorem{thm}{Theorem}
\newtheorem{lem}[thm]{Lemma}
\newtheorem{cor}[thm]{Corollary}
\theoremstyle{definition}
\newtheorem*{remark}{Remark}

\usepackage{xpatch}
\xpatchcmd{\proof}{\itshape}{\normalfont\proofnameformat}{}{}
\newcommand{\proofnameformat}{}

\begin{document}

\renewcommand{\proofnameformat}{\bfseries}

\begin{center}
{\Large\textbf{Equidistribution of continued fraction convergents in $\mathrm{SL}(2,\mathbb{Z}_m)$ with an application to local discrepancy}}

\vspace{10mm}

\textbf{Bence Borda}

{\footnotesize Graz University of Technology

Steyrergasse 30, 8010 Graz, Austria

Email: \texttt{borda@math.tugraz.at}}

\vspace{5mm}

{\footnotesize \textbf{Keywords:} convergent denominators mod $m$, Gauss--Kuzmin problem,\\ limit law, invariance principle, irrational rotation}

\end{center}

\vspace{5mm}

\begin{abstract}
Consider the sequence of continued fraction convergents $p_n/q_n$ to a random irrational number. We study the distribution of the sequences $p_n \pmod{m}$ and $q_n \pmod{m}$ with a fixed modulus $m$, and more generally, the distribution of the $2 \times 2$ matrix with entries $p_{n-1}, p_n, q_{n-1}, q_n \pmod{m}$. Improving the strong law of large numbers due to Sz\"usz, Moeckel, Jager and Liardet, we establish the central limit theorem and the law of the iterated logarithm, as well as the weak and the almost sure invariance principles. As an application, we find the limit distribution of the maximum and the minimum of the Birkhoff sum for the irrational rotation with the indicator of an interval as test function. We also compute the normalizing constant in a classical limit law for the same Birkhoff sum due to Kesten, and dispel a misconception about its dependence on the test interval.
\end{abstract}

\section{Introduction}

The statistical properties of the continued fraction expansion $\alpha = [0;a_1,a_2,\ldots]$ of a random real number $\alpha \in [0,1]$ is a classical topic in metric number theory. The partial quotients $a_n$ form a weakly dependent sequence of random variables, consequently the sum $\sum_{n=1}^N f(a_n)$ satisfies various probabilistic limit theorems depending on the growth rate of the function $f: \mathbb{N} \to \mathbb{R}$. The asymptotic behavior of the convergents $p_n/q_n = [0;a_1,a_2,\ldots, a_n]$ is also well known. For instance, we have $\log q_n \sim \frac{\pi^2}{12 \log 2} n$ for a.e.\ $\alpha$, and $\log q_n$ even satisfies the central limit theorem (CLT) and the law of the iterated logarithm (LIL). We refer to the monograph \cite{IK} for a comprehensive survey.

The main subject of this paper is the distribution of the sequences $p_n \pmod{m}$ and $q_n \pmod{m}$ in $\mathbb{Z}_m$ with a fixed integer $m \ge 2$. One of the first results in the area is due to Sz\"usz \cite[Satz 3.3]{SZ}, who showed that for any $a \in \mathbb{Z}_m$,
\begin{equation}\label{qnmodmae}
\lim_{N \to \infty} \frac{1}{N} \sum_{n=1}^N \mathds{1}_{\{ q_n \equiv a \pmod{m} \}} = \frac{\prod_{p \mid \mathrm{gcd} (a,m)} \left( 1-\frac{1}{p} \right)}{m \prod_{p \mid m} \left( 1-\frac{1}{p^2} \right)} \quad \textrm{for a.e. } \alpha,
\end{equation}
where the products are over prime divisors. Contrary to what one might naively expect, certain residue classes are thus attained more frequently than others. For instance, only one third of all convergent denominators are even, while two thirds of them are odd.

The origin of the limit in \eqref{qnmodmae} becomes transparent when we consider the joint distribution of the quadruple $(p_{n-1},p_n,q_{n-1},q_n) \pmod{m}$, and work in the group $\mathrm{SL}(2,\mathbb{Z}_m)$. The recursions for $p_n$ and $q_n$, with the usual convention $p_0=0$, $q_0=1$, can be written in matrix form as
\[ \left( \begin{array}{cc} 0 & 1 \\ 1 & a_1 \end{array} \right) \left( \begin{array}{cc} 0 & 1 \\ 1 & a_2 \end{array} \right) \cdots \left( \begin{array}{cc} 0 & 1 \\ 1 & a_n \end{array} \right) = \left( \begin{array}{cc} p_{n-1} & p_n \\ q_{n-1} & q_n \end{array} \right) , \]
where the right-hand side has determinant $(-1)^n$. Taking the previous formula mod $m$ entrywise leads to the $2 \times 2$ matrix with entries in $\mathbb{Z}_m$
\[ P_n = \left( \begin{array}{cc} p_{n-1} & p_n \\ q_{n-1} & q_n \end{array} \right) \pmod{m} . \]
For the rest of the paper, let
\[ G_D= \left\{ \left( \begin{array}{cc} a & b \\ c & d \end{array} \right) \, : \, a,b,c,d \in \mathbb{Z}_m, \,\, ad-bc=D \right\}, \qquad D \in \mathbb{Z}_m^* = \{ a \in \mathbb{Z}_m \, : \, \mathrm{gcd} (a,m)=1 \} , \]
and let $G=G_1 \cup G_{-1}$. Note that $G_1=\mathrm{SL}(2,\mathbb{Z}_m)$ is a normal subgroup of $\cup_{D \in \mathbb{Z}_m^*} G_D = \mathrm{GL}(2,\mathbb{Z}_m)$, and $G_D$ is a coset of $G_1$. If $m=2$, then $G_1=G_{-1}=\mathrm{SL}(2,\mathbb{Z}_2)=\mathrm{GL}(2,\mathbb{Z}_2)$, otherwise $G_1 \neq G_{-1}$.

It turns out that the sequence $P_n$ equidistributes in $G$, that is, for any $g \in G$,
\begin{equation}\label{Pnae}
\lim_{N \to \infty} \frac{1}{N} \sum_{n=1}^N \mathds{1}_{\{ P_n =g \}} = \frac{1}{|G|} \quad \textrm{for a.e. } \alpha .
\end{equation}
The set of possible rows or columns of matrices in $G_D$ is
\[ V = \{ (a,b) \in \mathbb{Z}_m^2 \, : \, \mathrm{gcd} (a,b,m)=1 \} . \]
In fact, both rows and both columns of a matrix that is uniformly distributed on $G_D$ with some $D \in \mathbb{Z}_m^*$ are uniformly distributed on $V$, see Lemma \ref{SL2Zmlemma} below. Relation \eqref{Pnae} thus immediately implies that for any $(a,b) \in V$,
\begin{equation}\label{pnqnmodmae}
\lim_{N \to \infty} \frac{1}{N} \sum_{n=1}^N \mathds{1}_{\{ (p_n,q_n) \equiv (a,b) \pmod{m} \}} = \frac{1}{|V|} \quad \textrm{for a.e. } \alpha ,
\end{equation}
and the same holds for the pairs $(q_{n-1}, q_n)$ and $(p_{n-1}, p_n)$. For instance, the ``parity type'' of the fraction $p_n/q_n$ can be even/odd, odd/even or odd/odd; note that the type even/even is impossible as $p_n$ and $q_n$ are coprime. The parity type of $p_n/q_n$ thus equidistributes in the set $\{ \textrm{even/odd, odd/even, odd/odd} \}$.

The four entries of a matrix that is uniformly distributed on $G_D$, however, are not uniformly distributed on $\mathbb{Z}_m$. Instead, the distribution of all four entries is the probability measure $\nu$ on $\mathbb{Z}_m$ which assigns the measure
\begin{equation}\label{nudef}
\nu_a = \frac{\prod_{p \mid \mathrm{gcd} (a,m)} \left( 1-\frac{1}{p} \right)}{m \prod_{p \mid m} \left( 1-\frac{1}{p^2} \right)}, \quad a \in \mathbb{Z}_m
\end{equation}
to the singleton $\{ a \}$, see Lemma \ref{SL2Zmlemma} below. This explains the value of the limit in \eqref{qnmodmae}, and shows that the same relation holds for $p_n$ as well.

Relation \eqref{pnqnmodmae} for the pair $(q_{n-1},q_n)$ was first proved by Sz\"usz \cite[Satz 3.2]{SZ} by following L\'evy's approach to the Gauss--Kuzmin problem about the mixing properties of the Gauss map, but his work seems to have gone mostly unnoticed by the ergodic theory community. Moeckel \cite{MO} used the close relationship between continued fractions and the geodesic flow on the modular surface, and almost proved \eqref{pnqnmodmae}; we say almost, as Moeckel worked with $\mathrm{PSL}(2,\mathbb{Z}_m)$ instead of $\mathrm{SL}(2,\mathbb{Z}_m)$, consequently he only showed equidistribution in the factor $V/\sim$ with the equivalence relation $(a,b) \sim (-a,-b)$. See \cite{FS} for a more recent account and generalizations of Moeckel's approach. Jager and Liardet \cite{JL} worked with the dynamical system $([0,1] \times G, \mu_{\mathrm{Gauss}} \otimes \mathrm{Unif}(G), S)$, where $\mu_{\mathrm{Gauss}}(A)=\frac{1}{\log 2} \int_A \frac{1}{1+x} \, \mathrm{d}x$, $A \subseteq [0,1]$ Borel, is the Gauss measure, $\mathrm{Unif}$ denotes the uniform probability measure on a finite set, and the transformation $S$ is the skew product over the Gauss map $Tx=\{ 1/x \}$ defined as
\begin{equation}\label{skewproduct}
S(x,g) = \left( Tx, g \left( \begin{array}{cc} 0 & 1 \\ 1 & a_1 \pmod{m} \end{array} \right) \right) ,
\end{equation}
with $a_1=\lfloor 1/x \rfloor$ the first partial quotient of $x$. Jager and Liardet proved that this system is ergodic, thereby establishing relation \eqref{Pnae} and consequently also \eqref{pnqnmodmae}.

In the terminology of probability theory, relations \eqref{qnmodmae}, \eqref{Pnae} and \eqref{pnqnmodmae} correspond to the strong law of large numbers. The main goal of the present paper is to extend these results to more precise limit theorems, such as the CLT and the LIL. In fact, we will even establish the weak and the almost sure invariance principles. We refer to Billingsley \cite{BI} for a general introduction to invariance principles and their relation to the ordinary and the functional CLT and LIL.

Counting problems in higher dimensional Diophantine approximation theory under similar congruence conditions have recently been investigated in \cite{AG1,AG2,AG3,AGY,NRS,SW} using dynamical methods. As a higher dimensional analogue of \eqref{pnqnmodmae}, the strong law of large numbers for the number of best rational approximation vectors under congruence conditions was established in \cite[Corollary 3.3]{SW}. Improving the latter result to the CLT and the LIL in the higher dimensional setting remains open.

We now state our main result. Throughout, $\alpha$ is a random variable with distribution $\mu$, a Borel probability measure on $[0,1]$ that is absolutely continuous with respect to the Lebesgue measure $\lambda$. In some of our results we assume that $\mu$ has a Lipschitz density, that is,
\begin{equation}\label{lipschitz}
\left| \frac{\mathrm{d}\mu}{\mathrm{d}\lambda} (x) - \frac{\mathrm{d}\mu}{\mathrm{d}\lambda} (y) \right| \le L|x-y| \quad \textrm{for all } x,y \in [0,1]
\end{equation}
with some constant $L \ge 0$. We write $X \sim \vartheta$ if the random variable $X$ has distribution $\vartheta$, and $\overset{d}{\to}$ denotes convergence in distribution.

Given an arbitrary function $f: G \to \mathbb{R}$, we define the constants $E_f$ and $\sigma_f \ge 0$ as follows. Let $U \sim \mathrm{Unif}(G)$, and define $E_f = \mathbb{E} (f(U))$. Let $\alpha \sim \mu_{\mathrm{Gauss}}$ and $U_{\pm 1} \sim \mathrm{Unif}(G_{\pm 1})$ be independent random variables, let $\bar{f}(x)=f(x)-\mathbb{E} (f(U_{\pm 1}))$ for $x \in G_{\pm 1}$, and define
\begin{equation}\label{sigmadef}
\sigma_f^2 = \frac{1}{2} \mathbb{E} (\bar{f}(U_1)^2) + \sum_{n=1}^{\infty} \mathbb{E} (\bar{f}(U_1) \bar{f}(U_1 P_n)) + \frac{1}{2} \mathbb{E} (\bar{f}(U_{-1})^2) + \sum_{n=1}^{\infty} \mathbb{E} (\bar{f}(U_{-1}) \bar{f}(U_{-1} P_n)) .
\end{equation}

\begin{thm}\label{Gtheorem} Fix an integer $m \ge 2$, and let $f: G \to \mathbb{R}$ be arbitrary.
\begin{enumerate}
\item[(i)] The right-hand side of \eqref{sigmadef} is finite and nonnegative.

\item[(ii)] Let $\alpha \sim \mu$ with $\mu \ll \lambda$. Then the process $\sum_{1 \le n \le tN} f(P_n)$, $t \in [0,1]$ satisfies the functional CLT
\[ \frac{\sum_{1 \le n \le tN} f(P_n) - E_f tN}{\sqrt{N}} \overset{d}{\to} \sigma_f W(t) \]
in the Skorokhod space $\mathcal{D}[0,1]$, where $W(t)$, $t \in [0,1]$ is a standard Wiener process.

\item[(iii)] Let $\alpha \sim \mu$ with $\mu \ll \lambda$, and assume \eqref{lipschitz}. Without changing its distribution, the process $\sum_{1 \le n \le t} f(P_n)$, $t \ge 0$ can be redefined on a richer probability space so that
\[ \sum_{1 \le n \le t} f(P_n) - E_f t = \sigma_f W(t) + O(t^{1/2-\eta}) \quad \textrm{a.s.} \]
with a universal constant $\eta>0$, where $W(t)$, $t \ge 0$ is a standard Wiener process.
\end{enumerate}
\end{thm}
\noindent Theorem \ref{Gtheorem} (ii) immediately implies that the sum $\sum_{n=1}^N f(P_n)$ satisfies the CLT
\[ \frac{\sum_{n=1}^N f(P_n) - E_f N}{\sqrt{N}} \overset{d}{\to} \mathcal{N}(0,\sigma_f^2) , \]
where $\mathcal{N}(a,\sigma^2)$ denotes the normal distribution with mean $a$ and variance $\sigma^2$ (or the constant $a$ in case $\sigma=0$), whereas Theorem \ref{Gtheorem} (iii) implies the LIL
\[ \limsup_{N \to \infty} \frac{\sum_{n=1}^N f(P_n) - E_f N}{\sqrt{2 N \log \log N}} = \sigma_f \quad \textrm{for a.e. } \alpha . \]
Theorem \ref{Gtheorem} applied to functions $f$ supported on $G_1$ resp.\ $G_{-1}$ describes the quantitative equidistribution of $P_{2n}$ resp.\ $P_{2n-1}$ in $G_1$ resp.\ $G_{-1}$. We conjecture that $\sigma_f=0$ if and only if both restrictions $f \mid_{G_{1}}$ and $f \mid_{G_{-1}}$ are constant functions (in which case the sequence $f(P_n)$ is deterministic), but this remains open.

Given an arbitrary function $f: V \to \mathbb{R}$, applying Theorem \ref{Gtheorem} to the map $\left( \begin{array}{cc} a & b \\ c & d \end{array} \right) \mapsto f(b,d)$ shows that the sum $\sum_{n=1}^N f((p_n,q_n) \pmod{m})$ satisfies the CLT
\[ \frac{\sum_{n=1}^N f((p_n,q_n) \pmod{m}) - E N}{\sqrt{N}} \overset{d}{\to} \mathcal{N} (0,\sigma^2) \]
and the LIL
\[ \limsup_{N \to \infty} \frac{\sum_{n=1}^N f((p_n,q_n) \pmod{m}) - E N}{\sqrt{2 N \log \log N}} = \sigma \quad \textrm{for a.e. } \alpha, \]
where $E=\mathbb{E} (f(U))$ with $U \sim \mathrm{Unif} (V)$, and the constant $\sigma \ge 0$ depends only on $f$. The same holds for the pairs $(q_{n-1}, q_n)$ and $(p_{n-1}, p_n)$.

Given an arbitrary function $f: \mathbb{Z}_m \to \mathbb{R}$, the correct centering term is defined using an auxiliary random variable $U \sim \nu$, that is, $E= \mathbb{E} (f(U)) = \sum_{a \in \mathbb{Z}_m} \nu_a f(a)$ with $\nu_a$ defined in \eqref{nudef}. Then the sum $\sum_{n=1}^N f(q_n \pmod{m})$ satisfies the CLT
\[ \frac{\sum_{n=1}^N f(q_n \pmod{m}) - EN}{\sqrt{N}} \overset{d}{\to} \mathcal{N}(0,\sigma^2) \]
and the LIL
\[ \limsup_{N \to \infty} \frac{\sum_{n=1}^N f(q_n \pmod{m}) - EN}{\sqrt{2 N \log \log N}} = \sigma \quad \textrm{for a.e. } \alpha \]
with a suitable constant $\sigma \ge 0$ depending only on $f$. The same holds for $p_n$.

The key ingredient in the proof of Theorem \ref{Gtheorem} is a version of the Gauss--Kuzmin--L\'evy theorem for the skew product \eqref{skewproduct}, see Theorem \ref{gausskuzmintheorem} below, which we prove using the Perron--Frobenius operator. As a corollary, in Lemma \ref{psilemma} we show that the sequence $P_n$ is $\psi$-mixing with exponential rate. The weak and the almost sure invariance principles then follow from general results of Philipp and Stout \cite{PS} on the partial sums of weakly dependent random variables. We refer to Liverani \cite{LI} for further applications of the Perron--Frobenius operator to mixing in dynamical systems.

In fact, we also prove that the sequence $(a_n,P_n)$ is $\psi$-mixing with exponential rate, from which various probabilistic limit theorems follow for the sum $\sum_{n=1}^N f(a_n,P_n)$ depending on the growth rate of $f$ in its first variable. As an application, in Section \ref{discrepancysection} we find the limit distribution of the maximum and the minimum of certain Birkhoff sums for the circle rotation. In Section \ref{SL2Zmsection}, we gather the necessary facts about the group $\mathrm{SL}(2,\mathbb{Z}_m)$. The proofs are given in Sections \ref{kestensection}, \ref{Pnsection} and \ref{limitlawsection}.

\section{An application to local discrepancy}\label{discrepancysection}

The circle rotation $x \mapsto x+\alpha \pmod{\mathbb{Z}}$ on $\mathbb{R}/\mathbb{Z}$ with a given irrational $\alpha$ is perhaps the simplest discrete time dynamical system. The system is uniquely ergodic, consequently Birkhoff sums satisfy $\sum_{n=1}^N f(n \alpha + \beta) = N \int_0^1 f(x) \, \mathrm{d}x+o(N)$ for any starting point $\beta$ and any $1$-periodic function $f$ that is Riemann integrable on $[0,1]$. The remainder term $o(N)$ depends sensitively on the continued fraction expansion of $\alpha$ and the function $f$, and satisfies various probabilistic limit theorems, see \cite{DF3} for a survey.

A classical example is $f(x)=\mathds{1}_{[0,r]}(\{ x \}) -r$ with a fixed $r \in (0,1)$, the centered indicator function of the interval $[0,r]$ extended with period $1$, where $\{ \cdot \}$ denotes the fractional part. The first limit theorem for the corresponding Birkhoff sum is due to Kesten \cite{KE1,KE2}, who proved that if $(\alpha, \beta)$ is a random variable uniformly distributed on the unit square, then
\begin{equation}\label{kesten1}
\frac{\sum_{n=1}^N \mathds{1}_{[0,r]} (\{ n \alpha + \beta \}) -rN}{\sigma \log N} \overset{d}{\to} \mathrm{Cauchy}.
\end{equation}
Here ``Cauchy'' denotes the standard Cauchy distribution, with density function $1/(\pi(1+x^2))$. Because of the random starting point $\beta$, the same holds for any subinterval of $[0,1]$ of length $r$.

Kesten gave a complicated but explicit formula for the normalizing constant $\sigma>0$, see Section \ref{kestensection}. He showed that the value of $\sigma$ is the same for all irrational $r$, but his formula involves $r$ if $r$ is rational. This apparent dependence of $\sigma$ on $r$ in the rational case has been cited by several authors. Disproving this long-held view, we show that the dependence is illusory, and \eqref{kesten1} holds with the same value of $\sigma$ for both rational and irrational $r$.
\begin{thm}\label{kestentheorem} Kesten's limit law \eqref{kesten1} holds with $\sigma=1/(3 \pi)$ for all $r \in (0,1)$.
\end{thm}

The function $f(x)=\{ x \} -1/2$ also leads to the same limit law \cite{KE1}: if $(\alpha, \beta)$ is a random variable uniformly distributed on the unit square, then
\begin{equation}\label{kesten2}
\frac{\sum_{n=1}^N (\{ n \alpha + \beta \} -1/2)}{\sigma' \log N} \overset{d}{\to} \mathrm{Cauchy}.
\end{equation}
In Section \ref{kestensection} we show that \eqref{kesten2} holds with $\sigma'=1/(4 \pi)$. The question whether \eqref{kesten1} and \eqref{kesten2} hold with a fixed starting point $\beta$ was already raised by Kesten, and remains open.

Remarkable random behavior of the Birkhoff sum $\sum_{n=1}^N f(n \alpha)$ with a fixed quadratic irrational $\alpha$, fixed starting point $\beta=0$ and test functions $f(x)=\mathds{1}_{[0,r]}(\{ x \}) -r$, $r \in (0,1)$ rational, and $f(x)=\{ x \} - 1/2$ was first proved by Beck \cite{BE}, and later generalized by Bromberg and Ulcigrai \cite{BU}. The special case of the former test function with $r=1/2$ is sometimes called the deterministic random walk \cite{AK,ADDS}. Higher dimensional analogues of \eqref{kesten1} are due to Dolgopyat and Fayad \cite{DF1,DF2}.

In this paper, we consider the Birkhoff sum $S_{N,r} (\alpha) = \sum_{n=1}^N \mathds{1}_{[0,r]} (\{ n \alpha \}) - rN$ with a rational $r$, that is, when the starting point is $\beta=0$, and we find the limit law of its maximum and minimum. Let $\mathrm{Stab}(1,\pm 1)$ denote the standard stable law of index $1$ and skewness parameter $\pm 1$, whose characteristic function is $\exp (-|x| (1 \pm (2i/\pi) \mathrm{sgn}(x) \log |x|))$, and let $\otimes$ denote the product of two measures.
\begin{thm}\label{discrepancytheorem} Let $\alpha \sim \mu$ with $\mu \ll \lambda$, and let $r \in (0,1)$ be a fixed rational. Then
\[ \left( \frac{\displaystyle{\max_{0 \le N<M} S_{N,r} (\alpha) -E_M}}{{\frac{1}{2 \pi} \log M}}, \frac{\displaystyle{\min_{0 \le N<M} S_{N,r} (\alpha) +E_M}}{{\frac{1}{2 \pi} \log M}} \right) \overset{d}{\to} \mathrm{Stab}(1,1) \otimes \mathrm{Stab}(1,-1) \quad \textrm{as } M \to \infty, \]
where $E_M = (1/\pi^2) \log M \log \log M - c(r) \log M$ with an explicit constant $c(r)$ depending only on $r$.
\end{thm}
\noindent In particular,
\[ \frac{\displaystyle{\max_{0 \le N<M} S_{N,r}(\alpha) -E_M}}{\frac{1}{2 \pi} \log M} \overset{d}{\to} \mathrm{Stab}(1,1) \qquad \textrm{and} \qquad \frac{\displaystyle{\min_{0 \le N<M} S_{N,r}(\alpha) +E_M}}{\frac{1}{2 \pi} \log M} \overset{d}{\to} \mathrm{Stab}(1,-1) . \]
The main motivation for deducing the joint limit law in Theorem \ref{discrepancytheorem} is that it immediately implies
\[ \frac{\displaystyle{\max_{0 \le N<M} |S_{N,r}(\alpha)| -E_M}}{\frac{1}{2 \pi} \log M} \overset{d}{\to} \vartheta , \]
where $\vartheta$ is the distribution of $\max \{ X,Y \}$, with $X,Y \sim \mathrm{Stab}(1,1)$ independent. Note that the cumulative distribution function of $\vartheta$ is the square of that of $\mathrm{Stab}(1,1)$.

The constant $c(r)$ in Theorem \ref{discrepancytheorem} depends only on the denominator of $r$ in its reduced form. If $r=l/m$ with some coprime integers $l$ and $m$, then $c(r) = (6 \theta_m + \gamma + \log (2 \pi) ) / \pi^2$, where $\gamma$ is the Euler--Mascheroni constant, and\footnote{Throughout the paper, we use the convention $0 \log 0 = 0$.}
\begin{equation}\label{thetam}
\theta_m = \sum_{a \in \mathbb{Z}_m} \nu_a \left\{ \frac{a}{m} \right\} \left( 1- \left\{ \frac{a}{m} \right\} \right) \log \left( \left\{ \frac{a}{m} \right\} \left( 1- \left\{ \frac{a}{m} \right\} \right) \right) .
\end{equation}

We rely on an explicit formula for the maximum and the minimum of $S_{N,r}(\alpha)$ in terms of the sequences $q_n \pmod{m}$ and $a_n$ due to Ro\c cadas and Schoissengeier \cite{SCH1}, see \eqref{schoissengeier} in Section \ref{limitlawsection}. Theorem \ref{discrepancytheorem} then follows from the fact that the sequence $(a_n,P_n)$ is $\psi$-mixing with exponential rate, as established in Lemma \ref{psilemma}. It would be interesting to see whether Theorem \ref{discrepancytheorem} holds with a fixed irrational $r$ as well. The maximum and the minimum of the Birkhoff sum $\sum_{n=1}^N (\{ n \alpha \} -1/2)$ satisfies the same limit law as in Theorem \ref{discrepancytheorem} with suitable normalizing constants, but the proof is based on the theory of quantum modular forms instead \cite{BO}.

Relation \eqref{kesten1} and Theorem \ref{discrepancytheorem} concern $|S_{N,r}(\alpha)|$, which is sometimes called the local discrepancy of the Kronecker sequence $\{ n \alpha \}$ at $r \in (0,1)$. Taking the supremum over all subintervals $I \subseteq [0,1]$ leads to the notion of discrepancy:
\[ D_N (\alpha) = \sup_{I \subseteq [0,1]} \left| \sum_{n=1}^N \mathds{1}_I (\{ n \alpha \}) - \lambda (I) N \right| . \]
Kesten \cite{KE3} proved that if $\alpha \sim \lambda$, then
\[ \begin{split} \frac{D_N (\alpha)}{\log N \log \log N} &\to \frac{2}{\pi^2} \quad \textrm{in measure, as } N \to \infty, \\ \frac{\displaystyle{\max_{0 \le N<M} D_N (\alpha)}}{\log M \log \log M} &\to \frac{3}{\pi^2} \quad \textrm{in measure, as } M \to \infty . \end{split} \]
See also \cite{SCH2}. It remains a challenging open problem to show whether
\[ \frac{D_N(\alpha) - (2/\pi^2) \log N \log \log N}{\log N} \quad \textrm{and} \quad \frac{\displaystyle{\max_{0 \le N <M} D_N(\alpha) - (3/\pi^2) \log M \log \log M}}{\log M} \]
have nondegenerate limit distributions.

\section{The group $\mathrm{SL}(2,\mathbb{Z}_m)$}\label{SL2Zmsection}

In this section, we recall some basic facts about the group $\mathrm{SL}(2,\mathbb{Z}_m)$. Fix an integer $m \ge 2$, let $G_D$ be the set of $2 \times 2$ matrices with entries in $\mathbb{Z}_m$ of determinant $D \in \mathbb{Z}_m^*$, and let $V=\{ (a,b) \in \mathbb{Z}_m \, : \, \mathrm{gcd}(a,b,m)=1 \}$, as in the Introduction. Let $\nu$ be the probability measure on $\mathbb{Z}_m$ defined in \eqref{nudef}.
\begin{lem}\label{SL2Zmlemma} Let $D \in \mathbb{Z}_m^*$. We have $|G_D| = m|V| = m^3 \prod_{p \mid m} \left( 1-\frac{1}{p^2} \right)$. If $U=\left( \begin{array}{cc} u_1 & u_2 \\ u_3 & u_4 \end{array} \right) \sim \mathrm{Unif} (G_D)$, then $(u_1,u_2), (u_3,u_4), (u_1,u_3), (u_2,u_4) \sim \mathrm{Unif} (V)$, and $u_1, u_2, u_3, u_4 \sim \nu$.
\end{lem}

\begin{proof} One readily checks that the value of $\mathrm{gcd} (a,b,m)$ is invariant under multiplying the row vector $(a,b) \in \mathbb{Z}_m^2$ by a matrix in $\mathrm{GL}(2,\mathbb{Z}_m)$ from the right. In particular, the group $\mathrm{GL}(2,\mathbb{Z}_m)$ acts on the set $V$ by multiplication.

First, let $D=1$. Recall that given $a,b,e \in \mathbb{Z}_m$, the linear congruence $ax+by=e$ has a solution $x,y \in \mathbb{Z}_m$ if and only if $\mathrm{gcd}(a,b,m) \mid e$. In particular, given any $(a,b) \in V$, there exist $c,d \in \mathbb{Z}_m$ such that $ad-bc=1$. Therefore we can obtain any $(a,b) \in V$ by multiplying the row vector $(1,0) \in V$ by a suitable matrix in $G_1$, showing that $G_1$ acts transitively on $V$. It immediately follows that if $U \sim \mathrm{Unif}(G_1)$, then its two rows are uniformly distributed on $V$. Since transposition is a bijection of $G_1$, the two columns of $U$ are also uniformly distributed on $V$.

The stabilizer of the row vector $(1,0) \in V$ under the action of $G_1$ is $\left\{ \left( \begin{array}{cc} 1 & 0 \\ a & 1 \end{array} \right) \, : \, a \in \mathbb{Z}_m \right\}$, which has size $m$. Hence $|G_1| = m |V|$.

Now fix $a \in \mathbb{Z}_m$. Since $\mathrm{gcd}(a,b,m) = \mathrm{gcd}(b, \mathrm{gcd}(a,m))$, we have
\[ |\{ b \in \mathbb{Z}_m \, : \, (a,b) \in V \}| = \frac{m}{\mathrm{gcd}(a,m)} \varphi (\mathrm{gcd}(a,m)) = m \prod_{p \mid \mathrm{gcd}(a,m)} \left( 1-\frac{1}{p} \right) , \]
where $\varphi$ is the Euler totient function. Given a divisor $d \mid m$, we have $|\{ a \in \mathbb{Z}_m \, : \, \mathrm{gcd}(a,m)=d \}| = \varphi (m/d)$, hence
\[ |V| = \sum_{a \in \mathbb{Z}_m} \frac{m}{\mathrm{gcd}(a,m)} \varphi (\mathrm{gcd}(a,m)) = \sum_{d \mid m} \frac{m}{d} \varphi (d) \varphi \left( \frac{m}{d} \right) = m^2 \prod_{p \mid m} \left( 1-\frac{1}{p^2} \right) . \]
The last step can be seen using the prime factorization of $m$. Since the rows and columns of $U$ are uniformly distributed on $V$, the previous two formulas show that each entry $u_i$ of $U$ attains $a \in \mathbb{Z}_m$ with probability $\nu_a$. This finishes the proof for the case $D=1$. The claims for general $D \in \mathbb{Z}_m^*$ follow immediately from the fact that $G_D$ is a coset of $G_1$ in $\mathrm{GL}(2,\mathbb{Z}_m)$.
\end{proof}

For the rest of the paper, let $m'=\prod_{p \mid m} p$ denote the radical (greatest square-free divisor) of $m$.
\begin{lem}\label{generatinglemma} We have
\[ \left\{ \prod_{i=1}^4 \left( \begin{array}{cc} 0 & 1 \\ 1 & x_i \end{array} \right) \, : \, x_1, x_2, x_3, x_4 \in \mathbb{Z}_m \right\} =G_1. \]
More precisely, for any $g \in G_1$ there exist at least $\varphi (m')$ elements $y \in \mathbb{Z}_{m'}$ such that whenever $x_4 \in \mathbb{Z}_m$ satisfies $x_4 \equiv y \pmod{m'}$, the equation $\prod_{i=1}^4 \left( \begin{array}{cc} 0 & 1 \\ 1 & x_i \end{array} \right) =g$ has a solution in the variables $x_1, x_2, x_3 \in \mathbb{Z}_m$.
\end{lem}

\begin{proof} Fix $a,b,c,d \in \mathbb{Z}_m$, $ad-bc=1$, and consider the equation $\prod_{i=1}^4 \left( \begin{array}{cc} 0 & 1 \\ 1 & x_i \end{array} \right) = \left( \begin{array}{cc} a & b \\ c & d \end{array} \right)$. Multiplying the four matrices, we find that this matrix equation is equivalent to the system of four equations
\begin{align*}
1+x_2 x_3 &= a, & x_2+x_4+x_2 x_3 x_4 &= b, \\ x_1 + x_3 + x_1 x_2 x_3 &= c, & 1+x_1 x_2 + x_1 x_4 + x_3 x_4 + x_1 x_2 x_3 x_4 &= d.
\end{align*}
Substituting $x_2 x_3= a-1$ into the second and third equation, we see that $x_2 = b-ax_4$ and $x_3 = c-ax_1$. We are left with two variables $x_1, x_4$, and using the assumption $ad-bc=1$, the system turns out to be equivalent to the single equation
\[ a x_1 x_4 - b x_1 - c x_4 +d-1=0 . \]
It will thus be enough to prove that there exist at least $\varphi (m')$ elements $y \in \mathbb{Z}_{m'}$ such that whenever $x_4 \in \mathbb{Z}_m$ satisfies $x_4 \equiv y \pmod{m'}$, then $ax_4-b \in \mathbb{Z}_m^*$. Indeed, then the remaining linear congruence $(ax_4 -b) x_1 = cx_4 -d+1$ has a solution in the variable $x_1 \in \mathbb{Z}_m$.

Given any $p \mid m$, the congruence $ax-b \equiv 0 \pmod{p}$ has at most one solution in the variable $x \in \mathbb{Z}_p$. Indeed, if $p \nmid a$, then this follows from the fact that $\mathbb{Z}_p$ is a field. If $p \mid a$, then the assumption $ad-bc=1$ implies that $p\nmid b$, and there are no solutions. By the Chinese remainder theorem, there are at least $\prod_{p \mid m} (p-1) = \varphi (m')$ elements $y \in \mathbb{Z}_{m'}$ such that $ay-b \not\equiv 0 \pmod{p}$ for all $p \mid m$. Whenever $x_4 \in \mathbb{Z}_m$ satisfies $x_4 \equiv y \pmod{m'}$, we have $ax_4-b \not\equiv 0 \pmod{p}$ for all $p \mid m$, and in particular, $ax_4-b \in \mathbb{Z}_m^*$.
\end{proof}

\begin{remark} It follows that
\[ \left\{ \prod_{i=1}^5 \left( \begin{array}{cc} 0 & 1 \\ 1 & x_i \end{array} \right) \, : \, x_1, x_2, x_3, x_4, x_5 \in \mathbb{Z}_m \right\} =G_{-1}. \]
In fact, we could even prescribe the value of $x_5$. The number of factors 4 resp.\ 5 needed to generate $G_1$ resp.\ $G_{-1}$ is sharp: it is easy to check that whenever $m>2$,
\[ \left\{ \prod_{i=1}^2 \left( \begin{array}{cc} 0 & 1 \\ 1 & x_i \end{array} \right) \, : \, x_1, x_2 \in \mathbb{Z}_m \right\} \neq G_1 \quad \textrm{and} \quad \left\{ \prod_{i=1}^3 \left( \begin{array}{cc} 0 & 1 \\ 1 & x_i \end{array} \right) \, : \, x_1, x_2, x_3 \in \mathbb{Z}_m \right\} \neq G_{-1} . \]
For the sake of completeness, we mention that in the case $m=2$ the group $G_1=G_{-1}$ can be generated using $3$ factors.
\end{remark}

\section{The normalizing constants in Kesten's limit laws}\label{kestensection}

Kesten \cite[Section 4]{KE2} gave the following explicit formula for the value of $\sigma$ in \eqref{kesten1}. Let
\[ V(x,u,y) = \frac{2}{\pi^2} \sum_{k=1}^{\infty} \frac{\sin (2 \pi k x) \sin (\pi k u) \cos (2 \pi k y)}{k^2} , \]
and let $F(u) = \int_0^1 \int_0^1 |V(x,u,y)| \, \mathrm{d}x \, \mathrm{d}y$. Then
\[ \sigma = \left\{ \begin{array}{ll} (6/\pi) \sum_{a \in \mathbb{Z}_m} \nu_a F (al/m) & \textrm{if $r=l/m$ with some coprime integers $l$ and $m$}, \\ (6/\pi) \int_0^1 F(u) \, \mathrm{d} u & \textrm{if $r$ is irrational}. \end{array} \right. \]
Kesten even gave the hint to use the classical Fourier series
\begin{equation}\label{fourierseries}
\sum_{k=1}^{\infty} \frac{\cos (2 \pi k x)}{2 \pi^2 k^2} = \frac{1}{2} \{ x \}^2 - \frac{1}{2} \{ x \} + \frac{1}{12}
\end{equation}
in order to compute $F(u)$, although he did not follow through on his own advice. This is exactly what we shall do.

\begin{proof}[Proof of Theorem \ref{kestentheorem}] Trigonometric identities show that
\[ \begin{split} &4 \sin (2 \pi k x) \sin (\pi k u) \cos (2 \pi k y) = \\ &\cos (2 \pi k (x+y-u/2)) + \cos (2 \pi k (x-y-u/2)) - \cos (2 \pi k (x+y+u/2)) - \cos (2 \pi k (x-y+u/2)) . \end{split} \]
Letting $B(x) = \{ x \}^2 /2 - \{ x \} /2 + 1/12$ denote the second Bernoulli polynomial, \eqref{fourierseries} thus gives
\[ V(x,u,y) = B(x+y-u/2) + B(x-y-u/2) - B(x+y+u/2) - B(x-y+u/2) . \]
Using the fact that $(x,y) \mapsto (x+y,x-y)$ is a measure preserving map of $\mathbb{R}^2 / \mathbb{Z}^2$, we obtain
\[ \begin{split} F(u) &= \int_0^1 \int_0^1 |B(x-u/2) + B(y-u/2) - B(x+u/2) - B(y+u/2)| \, \mathrm{d} x \, \mathrm{d} y \\ &= \int_0^1 \int_0^1 |B(x) + B(y) - B(x+u) - B(y+u)| \, \mathrm{d} x \, \mathrm{d} y . \end{split} \]
In particular, $F(u)$ is $1$-periodic. Assuming now $x,y,u \in [0,1)$, we have
\[ B(x)-B(x+u) = \frac{1}{2}(\{ x \} - \{ x+u \})(\{ x \} + \{ x+u \} -1) = \left\{ \begin{array}{ll} - \frac{u}{2} (2x+u-1) & \textrm{if } 0 \le x < 1-u, \\ \frac{1-u}{2} (2x+u-2) & \textrm{if } 1-u \le x < 1 , \end{array} \right. \]
and a similar formula holds for $B(y)-B(y+u)$. Elementary calculations yield
\[ \begin{split} \int_0^{1-u} \int_0^{1-u} |B(x) + B(y) - B(x+u) - B(y+u)| \, \mathrm{d} x \, \mathrm{d} y &= u \int_0^{1-u} \int_0^{1-u} |x+y+u-1| \, \mathrm{d} x \, \mathrm{d} y \\ &= \frac{u(1-u)^3}{3}, \end{split} \]
and
\[ \begin{split} \int_{1-u}^1 \int_{1-u}^1 |B(x) + B(y) - B(x+u) - B(y+u)| \, \mathrm{d} x \, \mathrm{d} y &= (1-u) \int_{1-u}^1 \int_{1-u}^1 |x+y+u-2| \, \mathrm{d} x \, \mathrm{d} y \\ &= \frac{u^3(1-u)}{3}. \end{split} \]
Further,
\[ \begin{split} \int_0^{1-u} \int_{1-u}^1 |B(x) + B(y) - B(x+u) - B(y+u)| \, \mathrm{d} x \, \mathrm{d} y &=  \int_0^{1-u} \int_{1-u}^1 |(1-u)x -uy-(1-u)^2| \, \mathrm{d}x \, \mathrm{d}y \\ &=\frac{u^2 (1-u)^2}{3}, \end{split} \]
and by symmetry the integral on $[0,1-u] \times [1-u,1]$ is the same. Therefore
\[ F(u) = \frac{u(1-u)^3}{3} + \frac{u^3(1-u)}{3} + \frac{2 u^2 (1-u)^2}{3} = \frac{u(1-u)}{3}, \quad u \in [0,1), \]
and by periodicity, $F(u)=\{ u \} (1-\{ u\})/3$, $u \in \mathbb{R}$. For irrational $r \in (0,1)$ we thus have
\[ \sigma = \frac{6}{\pi} \int_0^1 \frac{u(1-u)}{3} \, \mathrm{d} u = \frac{1}{3 \pi}, \]
as claimed.

Now let $r=l/m$ with some coprime integers $l$ and $m$. Since $\nu_a$ is invariant under multiplication by $l \in \mathbb{Z}_m^*$, we have
\[ \sigma = \frac{2}{\pi} \sum_{a \in \mathbb{Z}_m} \nu_a \left\{ \frac{a}{m} \right\} \left( 1- \left\{ \frac{a}{m} \right\} \right) . \]
It remains to show that $\sigma$ does not depend on the denominator $m$ either.

Let $\mathcal{P}$ denote the set of prime divisors of $m$. For any $I \subseteq \mathcal{P}$, let
\[ w(I)=\prod_{p \in I} \left( 1 - \frac{1}{p} \right) = \sum_{J \subseteq I} \frac{(-1)^{|J|}}{\prod_{p \in J} p}, \quad \textrm{and} \quad s(I)=\sum_{\substack{a \in \mathbb{Z}_m \\ \prod_{p \in I} p \mid a}} \left\{ \frac{a}{m} \right\} \left( 1- \left\{ \frac{a}{m} \right\} \right) . \]
In particular, $\nu_a = w(I_a)/(m \prod_{p \mid m} (1-1/p^2))$ with $I_a=\{ p \in \mathcal{P} \, : \, p \mid a \}$. We now apply an inclusion-exclusion type argument. Observe that given any pair of sets $J \subseteq K \subseteq \mathcal{P}$, we have
\[ \sum_{J \subseteq I \subseteq K} (-1)^{|I \backslash J|} = \sum_{M \subseteq K \backslash J} (-1)^{|M|} = \left\{ \begin{array}{ll} 1 & \textrm{if } J=K, \\ 0 & \textrm{if } J \neq K. \end{array} \right. \]
Consequently for any $K \subseteq \mathcal{P}$,
\[ \sum_{J \subseteq K} w(J) \sum_{J \subseteq I \subseteq K} (-1)^{|I \backslash J|} = w(K) . \]
Applying the previous formula to $K=I_a$ and changing the order of summations lead to
\begin{equation}\label{inclusionexclusion}
\begin{split} \sum_{a \in \mathbb{Z}_m} \nu_a \left\{ \frac{a}{m} \right\} &\left( 1- \left\{ \frac{a}{m} \right\} \right) \\ &= \frac{1}{m \prod_{p \mid m} \left( 1-\frac{1}{p^2} \right)} \sum_{a \in \mathbb{Z}_m} w(I_a) \left\{ \frac{a}{m} \right\} \left( 1- \left\{ \frac{a}{m} \right\} \right) \\ &= \frac{1}{m \prod_{p \mid m} \left( 1-\frac{1}{p^2} \right)} \sum_{a \in \mathbb{Z}_m} \sum_{J \subseteq I_a} \sum_{J \subseteq I \subseteq I_a} (-1)^{|I \backslash J|} w(J) \left\{ \frac{a}{m} \right\} \left( 1- \left\{ \frac{a}{m} \right\} \right) \\ &= \frac{1}{m \prod_{p \mid m} \left( 1-\frac{1}{p^2} \right)} \sum_{I \subseteq \mathcal{P}} \sum_{\substack{a \in \mathbb{Z}_m \\ I \subseteq I_a}} \left\{ \frac{a}{m} \right\} \left( 1- \left\{ \frac{a}{m} \right\} \right) \sum_{J \subseteq I} (-1)^{|I \backslash J|} w(J) \\ &= \frac{1}{m \prod_{p \mid m} \left( 1-\frac{1}{p^2} \right)} \sum_{I \subseteq \mathcal{P}} s(I) \sum_{J \subseteq I} (-1)^{|I \backslash J|} w(J) . \end{split}
\end{equation}
Observe that given any divisor $d \mid m$,
\[ \sum_{\substack{a \in \mathbb{Z}_m \\ d \mid a}} \left\{ \frac{a}{m} \right\} \left( 1- \left\{ \frac{a}{m} \right\} \right) = \sum_{k=1}^{m/d-1} \frac{kd}{m} \left( 1-\frac{kd}{m} \right) = \frac{1}{6} \left( \frac{m}{d} - \frac{d}{m} \right) . \]
In particular,
\[ s(I)= \frac{1}{6} \left( \frac{m}{\prod_{p \in I}p} - \frac{\prod_{p \in I} p}{m} \right) . \]
Another application of inclusion-exclusion shows that
\[ \sum_{J \subseteq I} (-1)^{|I \backslash J|} w(J) = \frac{(-1)^{|I|}}{\prod_{p \in I} p} . \]
The previous two formulas simplify \eqref{inclusionexclusion} to
\begin{equation}\label{sumovera}
\begin{split} \sum_{a \in \mathbb{Z}_m} \nu_a \left\{ \frac{a}{m} \right\} \left( 1- \left\{ \frac{a}{m} \right\} \right) &= \frac{1}{m \prod_{p \mid m} \left( 1-\frac{1}{p^2} \right)} \sum_{I \subseteq \mathcal{P}} \frac{1}{6} \left( \frac{m}{\prod_{p \in I}p} - \frac{\prod_{p \in I} p}{m} \right) \frac{(-1)^{|I|}}{\prod_{p \in I} p} \\ &= \frac{1}{6 \prod_{p \mid m} \left( 1-\frac{1}{p^2} \right)} \sum_{I \subseteq \mathcal{P}} \frac{(-1)^{|I|}}{\prod_{p \in I} p^2} = \frac{1}{6}. \end{split}
\end{equation}
Note that we used $\sum_{I \subseteq \mathcal{P}} (-1)^{|I|}=0$. Therefore $\sigma=1/(3 \pi)$ for rational $r$ as well, as claimed.
\end{proof}

The explicit formula for $\sigma'$ in \eqref{kesten2} is \cite{KE1}
\[ \sigma' = \frac{6}{\pi^3} \int_0^1 \int_0^1 \left| \sum_{k=1}^{\infty} \frac{\sin (2 \pi k x) \sin (2 \pi k y)}{k^2} \right| \, \mathrm{d} x \, \mathrm{d} y. \]
A trigonometric identity and the Fourier series \eqref{fourierseries} lead to
\[ \sum_{k=1}^{\infty} \frac{\sin (2 \pi k x) \sin (2 \pi k y)}{k^2} = \sum_{k=1}^{\infty} \frac{\cos (2 \pi k (x-y)) - \cos (2 \pi k (x+y))}{2k^2} = \pi^2 (B(x-y)-B(x+y)) , \]
hence
\[ \begin{split} \sigma' &= \frac{6}{\pi} \int_0^1 \int_0^1 |B(x-y)-B(x+y)| \, \mathrm{d}x \, \mathrm{d}y = \frac{6}{\pi} \int_0^1 \int_0^1 |B(x)-B(y)| \, \mathrm{d}x \, \mathrm{d}y \\ &= \frac{3}{\pi} \int_0^1 \int_0^1 |x^2-x-y^2+y| \, \mathrm{d}x \, \mathrm{d}y . \end{split} \]
The two diagonals partition the unit square into four right triangles. The sign of $x^2-x-y^2+y=(x-y)(x+y-1)$ is constant on these triangles, making them convenient domains for integrating the function $|x^2-x-y^2+y|$. Elementary calculations then yield $\sigma'=1/(4 \pi)$, as claimed.

\section{The sequence $P_n$}\label{Pnsection}

\subsection{Preliminaries}

Let $\mathcal{B}$ denote the family of Borel subsets of $[0,1]$. Writing each $\alpha \in [0,1]$ in continued fraction form $\alpha = [0;a_1,a_2,\ldots]$, the partial quotients $a_k$ are thus measurable functions on $([0,1], \mathcal{B})$. Let $\mathcal{A}_k^{\ell}$ denote the $\sigma$-algebra generated by $a_i$, $k \le i \le \ell$, and similarly let $\mathcal{A}_k^{\infty}$ be the $\sigma$-algebra generated by $a_i$, $i \ge k$. By convention, $\mathcal{A}_1^k$ with $k=0$ is the trivial $\sigma$-algebra. Let $T: [0,1) \to [0,1)$, $Tx=\{ 1/x \}$ if $x \neq 0$, $T0=0$ denote the Gauss map. Note that $\mathcal{A}_{k+1}^{\infty} = \{ T^{-k} B \, : \, B \in \mathcal{B} \}$, where $T^{-k}$ denotes the preimage with respect to the $k$th iterate of $T$.

Given a Borel probability measure $\mu$ on $[0,1]$, the partial quotients $a_k$ become random variables, and we interpret, say, $\mu ( a_k =5 ) = \mu (\{ \alpha \in [0,1] \, : \, a_k =5 \})$ as the probability of the event $a_k=5$. We also use the conditional probability notation $\mu (B \mid A) = \mu (A \cap B) / \mu (A)$. Let
\[ P_k = \prod_{i=1}^k \left( \begin{array}{cc} 0 & 1 \\ 1 & a_i \end{array} \right) \pmod{m} \qquad \textrm{and} \qquad P_{k,\ell} = \prod_{i=k}^{\ell} \left( \begin{array}{cc} 0 & 1 \\ 1 & a_i \end{array} \right) \pmod{m} , \]
with the convention that $P_{k+1,k}$ is the identity element of the group $G_1$. Let $p_k/q_k = [0;a_1,a_2,\ldots, a_k]$ be the convergents. Recall the identity $q_{k-1}/q_k=[0;a_k, a_{k-1}, \ldots, a_1]$.

We start with two preparatory lemmas.
\begin{lem}\label{lebesguelemma1} For any $A,A' \in \mathcal{A}_1^k$ and $B \in \mathcal{A}_{k+1}^{\infty}$ such that $\lambda (A), \lambda (A'), \lambda (B)>0$, we have
\[ \frac{1}{2} \le \frac{\lambda (B \mid A)}{\lambda (B \mid A')} \le 2. \]
In particular, $\frac{1}{2} \lambda (A) \lambda (B) \le \lambda (A \cap B) \le 2 \lambda (A) \lambda (B)$.
\end{lem}

\begin{proof} Let $b_1, b_2, \ldots, b_{k+n} \in \mathbb{N}$ be arbitrary, and let $p_{\ell}/q_{\ell}=[0;b_1,\ldots, b_{\ell}]$, $\ell =k-1,k$. The set $\{ a_1=b_1, a_2=b_2, \ldots, a_{k+n}=b_{k+n} \}$ is an interval with endpoints $\frac{p_k r + p_{k-1}}{q_k r +q_{k-1}}$ and $\frac{p_k r^* + p_{k-1}}{q_k r^* +q_{k-1}}$, where $r=[b_{k+1};b_{k+2}, \ldots, b_{k+n}]$ and $r^*=[b_{k+1};b_{k+2}, \ldots, b_{k+n}+1]$. Using the identity $q_k p_{k-1} - q_{k-1} p_k = (-1)^k$, the Lebesgue measure of this interval simplifies to
\[ \lambda (a_1=b_1, a_2=b_2, \ldots, a_{k+n}=b_{k+n}) = \frac{|r-r^*|}{(q_k r + q_{k-1})(q_k r^* + q_{k-1})} . \]
The set $\{ a_1=b_1, a_2=b_2, \ldots, a_k = b_k \}$ is also an interval, with endpoints $\frac{p_k+p_{k-1}}{q_k + q_{k-1}}$ and $\frac{p_k}{q_k}$, and has Lebesgue measure
\[ \lambda (a_1=b_1, a_2=b_2, \ldots, a_k = b_k) = \frac{1}{q_k (q_k + q_{k-1})} . \]
Hence
\[ \lambda (a_{k+1}=b_{k+1}, \ldots, a_{k+n}=b_{k+n} \mid a_1 = b_1, \ldots, a_k =b_k) = |r-r^*| \frac{1+\frac{q_{k-1}}{q_k}}{\left( r+\frac{q_{k-1}}{q_k} \right) \left( r^*+\frac{q_{k-1}}{q_k} \right)} . \]
Here $r,r^*$ do not depend on $b_1,\ldots, b_k$. Therefore for any $b_1', \ldots, b_k' \in \mathbb{N}$, we have
\[ \frac{\lambda (a_{k+1}=b_{k+1}, \ldots, a_{k+n}=b_{k+n} \mid a_1 = b_1, \ldots, a_k =b_k)}{\lambda (a_{k+1}=b_{k+1}, \ldots, a_{k+n}=b_{k+n} \mid a_1 = b_1', \ldots, a_k =b_k')} = \frac{\frac{1+\frac{q_{k-1}}{q_k}}{\left( r+\frac{q_{k-1}}{q_k} \right) \left( r^*+\frac{q_{k-1}}{q_k} \right)}}{\frac{1+\frac{q_{k-1}'}{q_k'}}{\left( r+\frac{q_{k-1}'}{q_k'} \right) \left( r^*+\frac{q_{k-1}'}{q_k'} \right)}} , \]
where $p_{\ell}'/q_{\ell}' = [0;b_1', \ldots, b_{\ell}']$, $\ell =k-1,k$.

We claim that the right-hand side of the previous formula lies in the interval $[1/2,2]$. Letting
\[ g(x) = \frac{1+x}{\left( r+x \right) \left( r^*+x \right)}, \qquad x \in [0,1] , \]
it will be enough to show that $g(x)/g(y) \le 2$ for all $x,y \in [0,1]$. Note that $r,r^* \ge 1$, and that $g'(x) \ge 0 \Leftrightarrow -1-\sqrt{(r-1)(r^*-1)} \le x \le -1+\sqrt{(r-1)(r^*-1)}$.\\

\noindent\textbf{Case 1.} Assume that $(r-1)(r^*-1) \le 1$. Then $g(x)$ is decreasing on $[0,1]$, hence
\[ \frac{g(x)}{g(y)} \le \frac{g(0)}{g(1)} = \frac{(r+1)(r^*+1)}{2 r r^*} \le 2 . \]

\noindent\textbf{Case 2.} Assume that $1 < (r-1)(r^*-1) \le 2$. Then $g(x)$ attains its minimum at $x=1$ and its maximum at $x=-1+\sqrt{(r-1)(r^*-1)}$. Hence
\[ \begin{split} \frac{g(x)}{g(y)} &\le \frac{g(-1+\sqrt{(r-1)(r^*-1)})}{g(1)} = \frac{(r+1)(r^*+1)}{2 \left( \sqrt{r-1} + \sqrt{r^*-1} \right)^2} \\ &= \frac{(r-1)(r^*-1) + 2r + 2 r^*}{2(r-1+2\sqrt{(r-1)(r^*-1)} +r^*-1)} \le \frac{2+2r+2r^*}{2 r+2 r^*}  \le 2. \end{split} \]

\noindent\textbf{Case 3.} Assume that $2<(r-1)(r^*-1) \le 4$. Then $g(x)$ attains its minimum at $x=0$ and its maximum at $x=-1+\sqrt{(r-1)(r^*-1)}$. Hence
\[ \begin{split} \frac{g(x)}{g(y)} &\le \frac{g(-1+\sqrt{(r-1)(r^*-1)})}{g(0)} = \frac{rr^*}{\left( \sqrt{r-1} + \sqrt{r^*-1} \right)^2} \\ &= \frac{(r-1)(r^*-1) + r+r^*-1}{r-1+2\sqrt{(r-1)(r^*-1)}+r^*-1} \le \frac{r+r^*+3}{r+r^*+2\sqrt{2}-2} \le 2 . \end{split} \]

\noindent\textbf{Case 4.} Assume that $(r-1)(r^*-1)>4$. Then $g(x)$ is increasing on $[0,1]$, hence
\[ \frac{g(x)}{g(y)} \le \frac{g(1)}{g(0)} = \frac{2rr^*}{(r+1)(r^*+1)} \le 2. \]
This finishes the proof of $g(x)/g(y) \le 2$ for all $x,y \in [0,1]$. In particular,
\[ \frac{1}{2} \le  \frac{\lambda (a_{k+1}=b_{k+1}, \ldots, a_{k+n}=b_{k+n} \mid a_1 = b_1, \ldots, a_k =b_k)}{\lambda (a_{k+1}=b_{k+1}, \ldots, a_{k+n}=b_{k+n} \mid a_1 = b_1', \ldots, a_k =b_k')} \le 2 . \]
By $\sigma$-additivity, for all $A,A' \in \mathcal{A}_1^k$ and all $B \in \mathcal{A}_{k+1}^{k+n}$ of positive Lebesgue measure,
\[ \frac{1}{2} \le \frac{\lambda(B \mid A)}{\lambda (B \mid A')} \le 2. \]
Since $\mathcal{A}_{k+1}^{k+n}$, $n \in \mathbb{N}$ generate $\mathcal{A}_{k+1}^{\infty}$, the same holds for all $B \in \mathcal{A}_{k+1}^{\infty}$ of positive Lebesgue measure.
\end{proof}

\begin{lem}\label{lebesguelemma2} \hspace{1mm}
\begin{enumerate}
\item[(i)] For any $0 \le n \le 3$ and $g \in G_{(-1)^n}$, we have either $\lambda (P_n=g)=0$ or $\lambda (P_n=g) \ge 1/(m+1)^6$.

\item[(ii)] For any $g \in G_1$, we have
\[ \lambda (P_4 =g) \ge \frac{\pi^6}{216} \cdot \frac{\varphi (m')}{(m+1)^6 (m'+1)(m'-\varphi (m')+1)} . \]

\item[(iii)] For any $n \ge 5$ and $g \in G_{(-1)^n}$, we have
\[ \lambda (P_n=g) \ge \frac{\pi^6}{432} \cdot \frac{\varphi (m')}{(m+1)^6 (m'+1)(m'-\varphi (m')+1)} . \]
\end{enumerate}
\end{lem}

\begin{proof} \textbf{(i)} Let $0 \le n \le 3$. Either $\lambda (P_n=g)=0$, or there exist positive integers $b_1, b_2, b_3 \le m$ such that
\[ \lambda (P_n =g) \ge \lambda (a_1 = b_1, a_2=b_2, a_3 = b_3) \ge \frac{1}{(m+1)^6} . \]

\noindent\textbf{(ii)} Let $g \in G_1$, and consider the equation
\begin{equation}\label{b1b2b3b4}
\prod_{i=1}^{4} \left( \begin{array}{cc} 0 & 1 \\ 1 & b_i \end{array} \right) \pmod{m} =g
\end{equation}
in the variables $b_1, b_2, b_3, b_4 \in \mathbb{N}$. We have
\[ \lambda (P_4=g) = \sum_{b_1, b_2, b_3, b_4} \lambda (a_1=b_1, a_2=b_2, a_3=b_3, a_4=b_4) = \sum_{b_1, b_2, b_3, b_4} \frac{1}{q_4 (q_4+q_3)}, \]
where the sums are over the set of solutions of the equation \eqref{b1b2b3b4}, and $q_4$ resp.\ $q_3$ is the denominator of $[0;b_1,b_2,b_3,b_4]$ resp.\ $[0;b_1,b_2,b_3]$. Let $\ell_1, \ell_2, \ell_3 \in \mathbb{N}$. Lemma \ref{generatinglemma} implies that there exist at least $\varphi (m')$ integers $1 \le b_4 \le m'$ for which the equation \eqref{b1b2b3b4} has an integer solution $(\ell_i-1) m < b_i \le \ell_i m$, $i=1,2,3$. One readily checks that for all such solutions of \eqref{b1b2b3b4}, we have $q_3 \le \ell_1 \ell_2 \ell_3 (m+1)^3$ and $q_4 \le \ell_1 \ell_2 \ell_3 (m+1)^3 b_4$. Hence
\[ \begin{split} \lambda (P_4=g) &\ge \sum_{b_4} \sum_{\ell_1, \ell_2, \ell_3 =1}^{\infty} \frac{1}{\ell_1^2 \ell_2^2 \ell_3^2 (m+1)^6 b_4 (b_4+1)} \ge \left( \frac{\pi^2}{6} \right)^3 \frac{1}{(m+1)^6} \sum_{j=m'-\varphi(m')+1}^{m'} \frac{1}{j(j+1)} \\ &= \frac{\pi^6}{216} \cdot \frac{\varphi (m')}{(m+1)^6 (m'+1)(m'-\varphi (m')+1)} . \end{split} \]

\noindent\textbf{(iii)} Let $n \ge 5$ and $g \in G_{(-1)^n}$. Lemma \ref{lebesguelemma1} and part (ii) show that
\[ \begin{split} \lambda (P_n=g) &= \sum_{h \in G_{(-1)^n}} \lambda (P_4 = gh^{-1}, P_{5,n}=h ) \ge  \sum_{h \in G_{(-1)^n}} \frac{1}{2} \lambda (P_4 = gh^{-1}) \lambda (P_{5,n}=h) \\ &\ge \frac{\pi^6}{432} \cdot \frac{\varphi (m')}{(m+1)^6 (m'+1)(m'-\varphi (m')+1)} . \end{split} \]
\end{proof}

\subsection{A Gauss--Kuzmin--L\'evy theorem}

In this section, we prove a version of the Gauss--Kuzmin--L\'evy theorem, which will serve as the main tool of this paper. We refer to \cite[Chapter 2]{IK} for a comprehensive account of the Gauss--Kuzmin problem. Theorem \ref{gausskuzmintheorem} below generalizes results of Kesten \cite[Lemma 2.5]{KE1} and Sz\"usz \cite{SZ}. They both considered the pair $(q_{k-1}, q_k) \pmod{m}$ and $\alpha \sim \lambda$, whereas we work with the matrix $P_k$, and allow the distribution of $\alpha$ to have a Lipschitz density.
\begin{thm}\label{gausskuzmintheorem} Let $\mu \ll \lambda$, and assume \eqref{lipschitz}. For any $A \in \mathcal{A}_1^k$, any $g \in G_{(-1)^n}$ and any $B \in \mathcal{A}_{k+n+1}^{\infty}$,
\[ \left| \mu \left( A \cap \{ P_{k+1,k+n} =g \} \cap B \right) - \frac{\mu(A) \mu_{\mathrm{Gauss}}(B)}{|G_1|} \right| \le C \lambda(A) \lambda (B) e^{-\tau n} , \]
where, with $m'=\prod_{p \mid m} p$ denoting the greatest square-free divisor of $m$,
\begin{equation}\label{Ctau}
C=4L+3 \quad \textrm{and} \quad \tau=\frac{\varphi (m')}{12 (m+1)^6 (m'+1) (m'-\varphi (m')+1)} \ge \frac{1}{12 (m+1)^8} .
\end{equation}
\end{thm}

\begin{proof} Throughout the proof we fix $k \ge 0$ and a set of the form $A=\{ a_1=b_1, \ldots, a_k =b_k \}$, with the convention $A=[0,1]$ if $k=0$. It will be enough to prove the theorem for this set $A$, as the claim for a general $A \in \mathcal{A}_1^k$ then follows by $\sigma$-additivity.

Given a Lipschitz function $F: [0,1] \to \mathbb{R}$, let
\[ \| F \|_{\mathrm{Lip}} = \sup_{\substack{x,y \in [0,1] \\ x \neq y}} \frac{|F(x)-F(y)|}{|x-y|} = \operatorname*{ess \, sup}_{x \in [0,1]} |F'(x)| \]
denote the Lipschitz constant. We implicitly use the fact that Lipschitz functions are a.e.\ differentiable, and satisfy the fundamental theorem of calculus.

For any $n \ge 0$ and $g \in G_{(-1)^n}$, the measure $B \mapsto \mu (A \cap \{ P_{k+1,k+n}=g \} \cap T^{-(k+n)}B)$, $B \in \mathcal{B}$ is absolutely continuous. Let $f_{n,g}$ denote its density with respect to $\lambda$. It will be enough to prove that
\begin{equation}\label{enough}
\sup_{\substack{g \in G_{(-1)^n} \\ x \in [0,1]}} \left| f_{n,g} (x) - \frac{\mu(A)}{|G_1| (\log 2)(1+x)} \right| \le C \lambda (A) e^{-\tau n} .
\end{equation}

Let $F_{n,g}(x)=(\log 2)(1+x) f_{n,g}(x)$, $x \in [0,1]$. We start with the case $n=0$.
\begin{lem}\label{F0lemma} For any $g \in G_1$, the functions $f_{0,g}$ and $F_{0,g}$ are Lipschitz, and we have
\[ \| F_{0,g} \|_{\mathrm{Lip}} \le (\log 2)(3L+2) \lambda(A) \quad \textrm{and} \quad 0 \le F_{0,g}(x) \le (\log 2) (L+2) \lambda(A) . \]
\end{lem}

\begin{proof} Assumption \eqref{lipschitz} implies that for all $x \in [0,1]$,
\[ \left| \frac{\mathrm{d}\mu}{\mathrm{d}\lambda} (x) -1 \right| = \left| \int_0^1 \left( \frac{\mathrm{d}\mu}{\mathrm{d}\lambda} (x) - \frac{\mathrm{d}\mu}{\mathrm{d}\lambda} (y) \right) \, \mathrm{d} y \right| \le \int_0^1 L |x-y| \, \mathrm{d} y \le \frac{L}{2} . \]
In particular, $\max_{[0,1]} \frac{\mathrm{d}\mu}{\mathrm{d}\lambda} \le \frac{L}{2}+1$.

By construction,
\[ \int_0^x f_{0,g}(t) \, \mathrm{d} t = \left\{ \begin{array}{ll} \mu (a_1=b_1, \ldots, a_k=b_k, [0;a_{k+1}, a_{k+2}, \ldots] \le x) & \textrm{if } g=1 \in G_1, \\ 0 & \textrm{if } g \neq 1 \in G_1. \end{array} \right. \]
The set $\{ a_1=b_1, \ldots, a_k=b_k, [0;a_{k+1}, a_{k+2}, \ldots] \le x \}$ is an interval with endpoints $p_k/q_k$ and $(p_k+p_{k-1}x) / (q_k + q_{k-1} x)$. Differentiating the previous formula with respect to $x$ thus gives
\[ f_{0,g} (x) = \left\{ \begin{array}{ll} \frac{\mathrm{d}\mu}{\mathrm{d} \lambda} \left( \frac{p_k+p_{k-1}x}{q_k +q_{k-1}x} \right) \frac{1}{(q_k+q_{k-1} x)^2} & \textrm{if } g=1 \in G_1, \\ 0 & \textrm{if } g \neq 1 \in G_1 . \end{array} \right. \]
Strictly speaking, the density function $f_{0,g}$ is only defined up to a.e.\ equivalence, and Lebesgue's differentiation theorem yields the previous formula for a.e.\ $x$. However, the right-hand side is a Lipschitz function by assumption, thus $f_{0,g}$ can be chosen to be Lipschitz. The previous formula thus holds for all $x \in [0,1]$, $F_{0,g}$ is Lipschitz, and for a.e.\ $x$ we have
\[ \begin{split} |F_{0,g}' (x)| &= (\log 2) \left| \left( \frac{\mathrm{d}\mu}{\mathrm{d}\lambda} \right)' \left( \frac{p_k+p_{k-1}x}{q_k +q_{k-1}x} \right) \frac{(-1)^k (1+x)}{(q_k + q_{k-1} x)^4} + \frac{\mathrm{d}\mu}{\mathrm{d}\lambda} \left( \frac{p_k+p_{k-1}x}{q_k +q_{k-1}x} \right) \frac{q_k-(2+x)q_{k-1}}{(q_k+q_{k-1}x)^3}  \right| \\ &\le \frac{\log 2}{(q_k+q_{k-1}x)^2} \left( L \frac{1+x}{q_k+q_{k-1}x} + \left( \frac{L}{2}+1 \right) \frac{|q_k-(2+x)q_{k-1}|}{q_k+q_{k-1}x} \right) \\ &\le \frac{\log 2}{q_k^2} \left( L \frac{2}{q_k+q_{k-1}} + \frac{L}{2}+1 \right) \le \frac{(\log 2)(3L+2)}{q_k (q_k+q_{k-1})} = (\log 2) (3L+2) \lambda (A) . \end{split} \]
In particular, $\| F_{0,g} \|_{\mathrm{Lip}} \le (\log 2)(3L+2) \lambda(A)$, as claimed. Further,
\[ 0 \le F_{0,g} (x) \le \max_{[0,1]} \frac{\mathrm{d}\mu}{\mathrm{d}\lambda} \cdot \frac{2 \log 2}{q_k (q_k+q_{k-1})} \le (\log 2) (L+2) \lambda (A) . \]
This finishes the proof of Lemma \ref{F0lemma}.
\end{proof}

We now prove that $f_{n,g}$ can be chosen to be Lipschitz by induction on $n$, the base case $n=0$ having been established in Lemma \ref{F0lemma}. Let $n \ge 1$ and $b \in \mathbb{N}$, and note that (ignoring endpoints)
\[ [0;a_{k+n+1}, a_{k+n+2}, \ldots] \le x \textrm{ and } a_{k+n}=b \, \Longleftrightarrow \, [0;a_{k+n}, a_{k+n+1}, \ldots] \in \left[ \frac{1}{b+x}, \frac{1}{b} \right] . \]
Using the partition $\{ a_{k+n}=b \}$, $b \in \mathbb{N}$ thus leads to
\[ \mu (A \cap \{ P_{k+1,k+n} =g \} \cap T^{-(k+n)} [0,x]) = \sum_{b=1}^{\infty} \mu \left( A \cap \{ P_{k+1,k+n-1} =gh(b)^{-1} \} \cap T^{-(k+n-1)} \left[ \frac{1}{b+x}, \frac{1}{b} \right] \right) , \]
where $h(b)=\left( \begin{array}{cc} 0 & 1 \\ 1 & b \end{array} \right) \pmod{m}$. Equivalently,
\[ \int_0^x f_{n,g} (t) \, \mathrm{d} t = \sum_{b=1}^{\infty} \int_{\frac{1}{b+x}}^{\frac{1}{b}} f_{n-1,gh(b)^{-1}} (t) \, \mathrm{d}t . \]
One readily checks that the series on the right-hand side can be differentiated term by term using the inductive hypothesis that $f_{n-1,gh(b)^{-1}}$ is Lipschitz, and we obtain
\[ f_{n,g}(x) = \sum_{b=1}^{\infty} f_{n-1, gh(b)^{-1}} \left( \frac{1}{b+x} \right) \frac{1}{(b+x)^2} . \]
The right-hand side is easily seen to be Lipschitz, hence $f_{n,g}$ can be chosen to be Lipschitz, and the previous formula holds for all $x \in [0,1]$. This finishes the induction. The recursion above should be compared to the Perron--Frobenius operator of the Gauss map \cite[Chapter 2]{IK}.

The recursion can be written in terms of $F_{n,g}$ as
\begin{equation}\label{recursion1}
F_{n,g}(x) = \sum_{b=1}^{\infty} F_{n-1,gh(b)^{-1}} \left( \frac{1}{b+x} \right) \frac{1+x}{(b+x)(b+1+x)} \qquad \textrm{for all } x \in [0,1] .
\end{equation}
Taking the derivative leads to
\begin{equation}\label{recursion2}
\begin{split} F_{n,g}'(x) = \sum_{b=1}^{\infty} \bigg( &F_{n-1,gh(b)^{-1}}' \left( \frac{1}{b+x} \right) \frac{-(1+x)}{(b+x)^3(b+1+x)} \\ &+ F_{n-1,gh(b)^{-1}} \left( \frac{1}{b+x} \right) \frac{b(b-1)-(1+x)^2}{(b+x)^2 (b+1+x)^2} \bigg) \qquad \textrm{for a.e. } x \in [0,1] . \end{split}
\end{equation}
For comparison, note the identities
\begin{equation}\label{identity}
\sum_{b=1}^{\infty} \frac{1+x}{(b+x)(b+1+x)} =1 \qquad \textrm{and} \qquad \sum_{b=1}^{\infty} \frac{b(b-1)-(1+x)^2}{(b+x)^2 (b+1+x)^2} =0.
\end{equation}
Indeed, the first series is telescoping, and the second identity follows from the first one via term by term differentiation. Define $L_n = \max_{g \in G_{(-1)^n}} \| F_{n,g} \|_{\mathrm{Lip}}$, and
\[ W_n^{-}= \min_{\substack{g \in G_{(-1)^n} \\ x \in [0,1]}} F_{n,g}(x),  \qquad W_n^{+}= \max_{\substack{g \in G_{(-1)^n} \\ x \in [0,1]}} F_{n,g}(x), \qquad \delta_n = W_n^{+}-W_n^{-} . \]
The recursion \eqref{recursion1} and the first identity in \eqref{identity} immediately show that $W_{n-1}^{-} \le W_n^{-}$ and $W_n^{+} \le W_{n-1}^{+}$, hence $\delta_n \le \delta_{n-1}$.

\begin{lem}\label{Lnlemma} We have $L_n \le (1-\zeta(2)+\zeta(3)) L_{n-1} + (1/4) \delta_{n-1}$, where $\zeta$ is the Riemann zeta function. The sum of the coefficients $1-\zeta(2)+\zeta(3)=0.5571\ldots$ and $1/4$ is less than $1$.
\end{lem}

\begin{proof} The recursion \eqref{recursion2} and the second identity in \eqref{identity} show that
\[ |F_{n,g}'(x)| \le L_{n-1} \sum_{b=1}^{\infty} \frac{1+x}{(b+x)^3(b+1+x)} + \sum_{b=1}^{\infty} \left| F_{n-1,gh(b)^{-1}} \left( \frac{1}{b+x} \right) -c \right| \frac{|b(b-1)-(1+x)^2|}{(b+x)^2(b+1+x)^2} \]
with any $c=c(g,x,y)$ which does not depend on $b$. Choosing $c=F_{n-1,gh(1)^{-1}} (1/(1+x))$, the $b=1$ term cancels, and by the definition of $\delta_{n-1}$ we obtain
\[ |F_{n,g}'(x)| \le L_{n-1} \sum_{b=1}^{\infty} \frac{1+x}{(b+x)^3(b+1+x)} + \delta_{n-1} \sum_{b=2}^{\infty} \frac{|b(b-1)-(1+x)^2|}{(b+x)^2(b+1+x)^2} . \]
The derivative of the first series is
\[ \left( \sum_{b=1}^{\infty} \frac{1+x}{(b+x)^3(b+1+x)} \right)' = -\frac{3x+5}{(1+x)^3 (2+x)^2} + \sum_{b=2}^{\infty} \frac{b^2-3b-3-(2b+6)x - 3x^2}{(b+x)^4 (b+1+x)^2} . \]
One readily checks that $-(3x+5)/((1+x)^3(2+x)^2)$ is increasing on $[0,1]$, and that the terms $b=2$ and $b=3$ in the series in the previous formula are negative for all $x \in [0,1]$. In particular,
\[ \left( \sum_{b=1}^{\infty} \frac{1+x}{(b+x)^3(b+1+x)} \right)' \le - \frac{1}{9} + \sum_{b=4}^{\infty} \frac{b^2 -3b-3}{b^4 (b+1)^2} =-0.1098\ldots <0 . \]
Therefore the maximum is attained at $x=0$, and
\[ \sum_{b=1}^{\infty} \frac{1+x}{(b+x)^3(b+1+x)} \le \sum_{b=1}^{\infty} \frac{1}{b^3(b+1)} = \sum_{b=1}^{\infty} \left( \frac{1}{b(b+1)} - \frac{1}{b^2} + \frac{1}{b^3} \right) = 1-\zeta(2)+\zeta(3) . \]
Considering $b=2$ and $b \ge 3$ separately, we check that each term in the second series attains its maximum at $x=0$ as well, hence
\[ \sum_{b=2}^{\infty} \frac{|b(b-1)-(1+x)^2|}{(b+x)^2(b+1+x)^2} \le \sum_{b=2}^{\infty} \frac{b^2-b-1}{b^2 (b+1)^2} = \sum_{b=2}^{\infty} \left( \frac{1}{b(b+1)} - \frac{1}{b^2} + \frac{1}{(b+1)^2} \right) = \frac{1}{4} . \]
This finishes the proof of Lemma \ref{Lnlemma}.
\end{proof}

For the sake of readability, let $\kappa_m = \varphi (m')/((m+1)^6 (m'+1)(m'-\varphi (m')+1))$.
\begin{lem}\label{deltanlemma} We have
\[ \delta_n \le \left( 1-\frac{\pi^6}{432} \kappa_m \right) \delta_{n-4} + \frac{\pi^6}{432} \kappa_m L_{n-4} . \]
\end{lem}

\begin{proof} If $L_{n-4} \ge \delta_{n-4}$, then the desired upper bound is greater or equal than $\delta_{n-4}$, and the claim follows from the fact that $\delta_n$ is nonincreasing. We may thus assume that $L_{n-4} < \delta_{n-4}$. 

Let $g \in G_1$ and $B \in \mathcal{B}$ be arbitrary. The partition $\{ P_{k+n-3,k+n}=h \}$, $h \in G_1$ leads to
\[ \mu ( A \cap \{ P_{k+1,k+n}=g \} \cap T^{-(k+n)} B ) = \sum_{h \in G_1} \mu (A \cap \{ P_{k+1,k+n-4} =gh^{-1} \} \cap T^{-(k+n-4)} (\{ P_4 =h \} \cap T^{-4} B) ) . \]
By the definition of $f_{n,g}$ and $F_{n,g}$, this is equivalent to
\[ \int_B F_{n,g} (x) \, \mathrm{d}\mu_{\mathrm{Gauss}}(x) = \sum_{h \in G_1} \int_{\{ P_4=h \} \cap T^{-4}B} F_{n-4,gh^{-1}} (x) \, \mathrm{d}\mu_{\mathrm{Gauss}}(x) . \]
Fix a pair $(g_0,x_0) \in G_1 \times [0,1]$ at which the minimum $F_{n-4,g_0}(x_0) = W_{n-4}^{-}$ is attained. Using the bound $F_{n-4,gh^{-1}}(x) \le W_{n-4}^{+}$ for all $h \neq g_0^{-1}g$ yields
\[ \int_B F_{n,g} (x) \, \mathrm{d}\mu_{\mathrm{Gauss}}(x) \le W_{n-4}^{+} \mu_{\mathrm{Gauss}} (B) + \int_{\{ P_4=g_0^{-1}g \} \cap T^{-4} B} \left( F_{n-4,g_0} (x) - W_{n-4}^{+} \right) \, \mathrm{d} \mu_{\mathrm{Gauss}} (x) . \]
Here $F_{n-4,g_0}(x) \le F_{n-4,g_0}(x_0) + L_{n-4} |x-x_0| \le W_{n-4}^{-} + L_{n-4}$, thus
\[ \int_B F_{n,g} (x) \, \mathrm{d}\mu_{\mathrm{Gauss}}(x) \le W_{n-4}^{+} \mu_{\mathrm{Gauss}} (B) + (L_{n-4} - \delta_{n-4}) \mu_{\mathrm{Gauss}} (\{ P_4=g_0^{-1}g \} \cap T^{-4} B) . \]
Note that $L_{n-4} - \delta_{n-4}<0$ by assumption. Lemmas \ref{lebesguelemma1} and \ref{lebesguelemma2} show that
\[ \lambda (\{ P_4=g_0^{-1}g \} \cap T^{-4} B) \ge \frac{1}{2} \lambda (P_4=g_0^{-1}g) \lambda (T^{-4}B) \ge \frac{\pi^6}{432} \kappa_m \lambda (T^{-4}B) , \]
and as the density of $\mu_{\mathrm{Gauss}}$ lies between $1/(2 \log 2)$ and $1/\log 2$,
\[ \mu_{\mathrm{Gauss}} (\{ P_4=g_0^{-1}g \} \cap T^{-4} B) \ge \frac{\pi^6}{864} \kappa_m \mu_{\mathrm{Gauss}} (B) . \]
Therefore
\[  \int_B F_{n,g} (x) \, \mathrm{d}\mu_{\mathrm{Gauss}}(x) \le W_{n-4}^{+} \mu_{\mathrm{Gauss}} (B) + (L_{n-4} - \delta_{n-4})  \frac{\pi^6}{864} \kappa_m \mu_{\mathrm{Gauss}} (B) . \]
A similar proof shows the lower bound
\[ \int_B F_{n,g} (x) \, \mathrm{d}\mu_{\mathrm{Gauss}}(x) \ge W_{n-4}^{-} \mu_{\mathrm{Gauss}} (B) - (L_{n-4} - \delta_{n-4}) \frac{\pi^6}{864} \kappa_m \mu_{\mathrm{Gauss}} (B) . \]
As the previous two formulas hold for all $B \in \mathcal{B}$, we have
\[  W_{n-4}^{-} - (L_{n-4} - \delta_{n-4}) \frac{\pi^6}{864} \kappa_m \le F_{n,g} (x) \le W_{n-4}^{+} +(L_{n-4} - \delta_{n-4})  \frac{\pi^6}{864} \kappa_m , \]
hence
\[ \delta_n \le \delta_{n-4} + (L_{n-4} - \delta_{n-4})  \frac{\pi^6}{432} \kappa_m . \]
This finishes the proof of Lemma \ref{deltanlemma}.
\end{proof}

Let $z=1-\zeta(2)+\zeta(3)$. Iterating Lemma \ref{Lnlemma} five times and using the fact that $\delta_n$ is nonincreasing shows that
\[ L_n \le z^5 L_{n-5} + (z^4+z^3+z^2+z+1)\frac{1}{4} \delta_{n-5}, \]
where the sum of the coefficients is $z^5+(z^4+z^3+z^2+z+1)/4= 0.5878\ldots$. Lemmas \ref{deltanlemma} and \ref{Lnlemma} yield
\[ \delta_n \le \left( 1-\frac{\pi^6}{432} \kappa_m \right) \delta_{n-5} + \frac{\pi^6}{432} \kappa_m \left( z L_{n-5} + \frac{1}{4} \delta_{n-5} \right) , \]
where the sum of the coefficients is
\[ 1-\frac{\pi^6}{432} \kappa_m + \frac{\pi^6}{432} \kappa_m \left( z+\frac{1}{4} \right) = 1-0.08584\ldots \cdot (5 \kappa_m) > 0.5879 . \]
We thus have the recursive upper bound $\max\{ L_n, \delta_n \} \le (1-0.08584 \cdot (5 \kappa_m)) \max \{ L_{n-5}, \delta_{n-5} \}$. Lemma \ref{F0lemma} implies that $\max \{ L_i, \delta_i \} \le (\log 2) (3L+2) \lambda(A)$ for $i=0$, and four applications of Lemma \ref{Lnlemma} shows that the same holds for $i=1,2,3,4$. Iterating the recursive upper bound $\lfloor n/5 \rfloor$ times thus leads to
\[ \max \{ L_n, \delta_n \} \le \left( 1-0.08584 \cdot (5 \kappa_m) \right)^{\lfloor n/5 \rfloor} (\log 2) (3L+2) \lambda(A) \le C_0 e^{-0.08584 \kappa_m n} \]
with $C_0 = e^{0.08584 \cdot (5 \kappa_m)}(\log 2) (3L+2)$. By the definition of $\delta_n$, this means that $F_{n,g}(x)$ lies in a given interval of length $C_0 e^{-0.08584 \kappa_m n}$ for all $g \in G_{(-1)^n}$ and $x \in [0,1]$. Then the average value
\[ \frac{1}{|G_1|} \sum_{g \in G_{(-1)^n}} \int_0^1 F_{n,g}(x) \, \mathrm{d} \mu_{\mathrm{Gauss}} (x) = \frac{\mu(A)}{|G_1|} \]
lies in the same interval, hence $| F_{n,g}(x) - \mu(A)/|G_1| | \le C_0 e^{-0.08584 \kappa_m n}$. Formula \eqref{enough} follows with
\[ C= \frac{C_0}{\log 2} \le 4L+3 \quad \textrm{and} \quad \tau=0.08584 \kappa_m > \frac{\kappa_m}{12} . \]
This finishes the proof of Theorem \ref{gausskuzmintheorem}.
\end{proof}

We now show that the limit relation in Theorem \ref{gausskuzmintheorem} without the exponential rate remains true for an arbitrary absolutely continuous measure, without assuming that the density is Lipschitz.
\begin{cor}\label{gausskuzmincorollary} Let $\mu \ll \lambda$. Then
\[ \lim_{n \to \infty} \sup_{k \ge 0} \sup_{\substack{A \in \mathcal{A}_1^k, \,\, B \in \mathcal{A}_{k+n+1}^{\infty} \\ g \in G_{(-1)^n}}} \left| \mu \left( A \cap \{ P_{k+1,k+n}=g \} \cap B \right) - \frac{\mu(A) \mu_{\mathrm{Gauss}}(B)}{|G_1|} \right| =0 . \]
\end{cor}

\begin{proof} Let $f(x) = \frac{\mathrm{d}\mu}{\mathrm{d}\lambda} (x)$ denote the density function. Using a positive mollifier on the circle group $\mathbb{R}/\mathbb{Z}$, we deduce that for any $\varepsilon >0$ there exists a smooth function $f_{\varepsilon}$ on $[0,1]$ such that
\[ \int_0^1 f_{\varepsilon} (x) \, \mathrm{d} x =1, \qquad \int_0^1 |f(x)-f_{\varepsilon}(x)| \, \mathrm{d}x < \varepsilon, \qquad \min_{x \in [0,1]} f_{\varepsilon}(x) >0 . \]
Let $\mu_{\varepsilon}$ be the Borel probability measure on $[0,1]$ with density $f_{\varepsilon}$. In particular, $|\mu (A) - \mu_{\varepsilon} (A)|<\varepsilon$ for all Borel sets $A \subseteq [0,1]$. Since $\mu_{\varepsilon}$ has a positive smooth density, the claim holds for $\mu_{\varepsilon}$ by Theorem \ref{gausskuzmintheorem}. As $\varepsilon$ was arbitrary, the claim holds also for $\mu$.
\end{proof}

\subsection{Weak convergence and mixing properties of $P_n$}

Theorem \ref{gausskuzmintheorem} and Corollary \ref{gausskuzmincorollary} immediately imply that $P_{2n}$ resp.\ $P_{2n-1}$ converges to the uniform distribution on $G_1$ resp.\ $G_{-1}$.
\begin{cor}\label{weakconvergencecorollary} Let $\alpha \sim \mu$ with $\mu \ll \lambda$. Then $P_{2n} \overset{d}{\to} \mathrm{Unif}(G_1)$ and $P_{2n-1} \overset{d}{\to} \mathrm{Unif}(G_{-1})$. Under the assumption \eqref{lipschitz}, we also have $\max_{g \in G_{(-1)^n}} |\mu (P_n=g) - 1/|G_1|| \le C e^{-\tau n}$, where $C$ and $\tau$ are as in \eqref{Ctau}.
\end{cor}
\noindent By Lemma \ref{SL2Zmlemma} we thus have $(p_n,q_n) \pmod{m} \overset{d}{\to} \mathrm{Unif}(V)$, and the same holds for $(q_{n-1},q_n)$ and $(p_{n-1},p_n)$. Further, $q_n \pmod{m} \overset{d}{\to} \nu$ and $p_n \pmod{m} \overset{d}{\to} \nu$. If the density function $\frac{\mathrm{d} \mu}{\mathrm{d}\lambda}$ is Lipschitz, the same hold with exponential rate. See also \cite{DK} for the special case $\mu = \lambda$ of Corollary \ref{weakconvergencecorollary}.

Theorem \ref{gausskuzmintheorem} and Corollary \ref{gausskuzmincorollary} also imply certain mixing properties of $P_n$. Let us first recall two classical ways of quantifying mixing, see \cite{BR} for more context. Let $X_n$, $n \in \mathbb{N}$ be a sequence of random variables on a probability space $(\Omega, \mathcal{F}, P)$ taking values from a measurable space. Let $\mathcal{F}_k^{\ell}$ denote the $\sigma$-algebra generated by $X_i$, $k \le i \le \ell$, and similarly let $\mathcal{F}_k^{\infty}$ be the $\sigma$-algebra generated by $X_i$, $i \ge k$. The $\alpha$-mixing (or strong mixing) coefficients of the sequence $X_n$ are defined as\footnote{The term $\alpha$-mixing and the notation $\alpha(\ell)$ are unrelated to the random real number $\alpha \in [0,1]$.}
\[ \alpha (\ell) = \sup_{k \in \mathbb{N}} \sup_{\substack{A \in \mathcal{F}_1^k \\ B \in \mathcal{F}_{k+\ell}^{\infty}}} \left| P (A \cap B) - P (A) P (B) \right|, \qquad \ell \in \mathbb{N} , \]
whereas the $\psi$-mixing coefficients are
\[ \psi (\ell) = \sup_{k \in \mathbb{N}} \sup_{\substack{A \in \mathcal{F}_1^k, \,\, P (A)>0 \\ B \in \mathcal{F}_{k+\ell}^{\infty}, \,\, P (B)>0}} \left| \frac{P(A \cap B)}{P(A) P(B)} -1 \right|, \qquad \ell \in \mathbb{N} . \]

\begin{lem}\label{psilemma} Let $\alpha \sim \mu$ with $\mu \ll \lambda$, and assume \eqref{lipschitz}. Then the $\alpha$-mixing coefficients of the sequence $X_n=(a_n,P_n)$ satisfy $\alpha(\ell) \ll e^{-\tau \ell}$, $\ell \in \mathbb{N}$, where $\tau$ is as in \eqref{Ctau}, and the implied constant depends only on $L$ and $m$. Under the additional assumption $\frac{\mathrm{d}\mu}{\mathrm{d}\lambda} (x) \ge K>0$ for all $x \in [0,1]$ with some constant $K>0$, we also have $\psi (\ell) \ll e^{-\tau \ell}$, $\ell \in \mathbb{N}$, with an implied constant depending also on $K$.
\end{lem}

\begin{remark} Exactly as in the proof of Corollary \ref{gausskuzmincorollary}, we can prove that under the sole assumption $\mu \ll \lambda$, we have $\lim_{\ell \to \infty} \alpha (\ell)=0$ without an estimate for the rate.
\end{remark}

\begin{proof}[Proof of Lemma \ref{psilemma}] Fix $k, \ell \ge 1$. Observe that $(X_1, X_2, \ldots, X_k)$ is a function of $(a_1, a_2, \ldots, a_k)$, and that $(X_{k+\ell}, X_{k+\ell+1}, \ldots)$ is a function of $(P_{k+\ell-1},a_{k+\ell}, a_{k+\ell+1}, \ldots )$. In particular, $\mathcal{F}_1^k \subseteq \mathcal{A}_1^k$, and any $B \in \mathcal{F}_{k+\ell}^{\infty}$ is of the form $B=\cup_{g \in G_{(-1)^{k+\ell-1}}} \{ P_{k+\ell-1} =g \} \cap B_g$ with $B_g \in \mathcal{A}_{k+\ell}^{\infty}$.

Let $A=\{ a_1=b_1, \ldots, a_k=b_k \}$, and let $h=\prod_{i=1}^k \left( \begin{array}{cc} 0 & 1 \\ 1 & b_i \end{array} \right) \pmod{m}$. An application of Theorem \ref{gausskuzmintheorem} shows that
\[ \mu (A \cap \{ P_{k+\ell-1}=g \} \cap B_g) = \mu (A \cap \{ P_{k+1, k+\ell-1}=h^{-1} g \} \cap B_g ) = \frac{\mu (A) \mu_{\mathrm{Gauss}}(B_g)}{|G_1|} +O( \lambda(A) \lambda (B_g) e^{-\tau \ell} ) . \]
By $\sigma$-additivity, the same holds with any $A \in \mathcal{A}_1^k$. In particular,
\[ \mu (\{ P_{k+\ell-1}=g \} \cap B_g) = \frac{\mu_{\mathrm{Gauss}}(B_g)}{|G_1|} +O(\lambda (B_g) e^{-\tau \ell} ) . \]
The previous two formulas and the fact that $\mu (A) \ll \lambda (A)$ yield
\[ \left| \mu (A \cap \{ P_{k+\ell-1}=g \} \cap B_g) - \mu(A) \mu (\{ P_{k+\ell-1}=g \} \cap B_g) \right| \ll \lambda (A) \lambda (B_g) e^{-\tau \ell} . \]
It is enough to sum over those $g$ for which $\lambda (P_{k+\ell-1}=g)>0$, as otherwise $\mu (P_{k+\ell-1}=g)=0$. Thus
\[ \left| \mu (A \cap B) - \mu(A) \mu (B) \right| \ll \lambda(A) e^{-\tau \ell} \sum_{\substack{g \in G_{(-1)^{k+\ell-1}} \\ \lambda (P_{k+\ell-1}=g)>0}} \lambda (B_g) . \]
Lemmas \ref{lebesguelemma1} and \ref{lebesguelemma2} show that each term on the right-hand side satisfies
\[ \lambda (B_g) \le 2 \frac{\lambda (\{ P_{k+\ell-1}=g \} \cap B_g)}{\lambda (P_{k+\ell-1}=g)} \ll \lambda (\{ P_{k+\ell-1}=g \} \cap B_g) . \]
Hence
\[ \sum_{\substack{g \in G_{(-1)^{k+\ell-1}} \\ \lambda (P_{k+\ell-1}=g)>0}} \lambda (B_g) \ll \lambda (B), \]
and we obtain $|\mu (A \cap B) - \mu(A) \mu (B)| \ll \lambda(A) \lambda(B) e^{-\tau \ell}$. In particular, $\alpha (\ell) \ll e^{-\tau \ell}$. Under the additional assumption $\frac{\mathrm{d}\mu}{\mathrm{d}\lambda} (x) \ge K>0$ we have $\lambda(A) \lambda(B) \ll \mu(A) \mu(B)$, and $\psi (\ell) \ll e^{-\tau \ell}$ follows.
\end{proof}

\subsection{Invariance principles for $P_n$}

We now compute the variance of the sum $\sum_{n=M+1}^{M+N} f(P_n)$, and then prove Theorem \ref{Gtheorem}.
\begin{lem}\label{variancelemma} Fix an integer $m \ge 2$, and let $f: G \to \mathbb{R}$ be arbitrary.
\begin{enumerate}
\item[(i)] The right-hand side of \eqref{sigmadef} is finite and nonnegative.

\item[(ii)] Let $\alpha \sim \mu$ with $\mu \ll \lambda$, and assume \eqref{lipschitz}. For any integers $M \ge 0$ and $N \ge 1$,
\[ \mathbb{E} \left( \sum_{n=M+1}^{M+N} f(P_n) - E_f N \right)^2 = \sigma_f^2 N + O( \log (N+1)) \]
with an implied constant depending only on $L$ and $f$.
\end{enumerate}
\end{lem}

\begin{proof} \textbf{(i)} Let $\alpha \sim \mu_{\mathrm{Gauss}}$ and $U_{\pm 1} \sim \mathrm{Unif} (G_{ \pm 1})$ be independent random variables. By Corollary \ref{weakconvergencecorollary},
\[ \left| \mathbb{E} (\bar{f}(U_{\pm 1}) \bar{f}(U_{\pm 1} P_n)) \right| = \left| \sum_{\substack{g \in G_{\pm 1} \\ h \in G_{(-1)^n}}} \bar{f}(g) \bar{f}(gh) \frac{\mu_{\mathrm{Gauss}} (P_n=h)}{|G_1|} \right| = \left| \sum_{\substack{g \in G_{\pm 1} \\ h \in G_{(-1)^n}}} \bar{f}(g) \bar{f}(gh) \frac{1}{|G_1|^2} \right| + O(e^{-\tau n}) . \]
Summing over $h \in G_{(-1)^n}$ and using that $\sum_{w \in G_{\pm 1}} \bar{f}(w)=0$ we see that the last sum vanishes. Hence $|\mathbb{E} (\bar{f}(U_{\pm 1}) \bar{f}(U_{\pm 1}P_n))| \ll e^{-\tau n}$, so both series in \eqref{sigmadef} are absolutely convergent. The fact that the right-hand side of \eqref{sigmadef} is nonnegative will follow from (ii).\\

\noindent\textbf{(ii)} We may assume that $\mathbb{E} (f(U_{\pm 1}))=0$. Expanding the square leads to
\begin{equation}\label{expandsquare}
\mathbb{E} \left( \sum_{n=M+1}^{M+N} f(P_n) \right)^2 = \sum_{n=M+1}^{M+N} \mathbb{E} (f(P_n)^2) + 2 \sum_{\ell =1}^{N-1} \sum_{n=M+1}^{M+N-\ell} \mathbb{E} (f(P_n) f(P_{n+\ell})) .
\end{equation}
Corollary \ref{weakconvergencecorollary} shows that $\mathbb{E} (f(P_n)^2) = \mathbb{E} (f(U_{(-1)^n})^2) + O(e^{-\tau n})$, hence the first sum in \eqref{expandsquare} is
\[ \sum_{n=M+1}^{M+N} \mathbb{E} (f(P_n)^2) = \left( \frac{1}{2} \mathbb{E} (f(U_1)^2) + \frac{1}{2} \mathbb{E} (f(U_{-1})^2) \right) N + O(1) . \]
Let $1 \le R \le N-1$ be a parameter to be chosen, and consider the second sum in \eqref{expandsquare}. Lemma \ref{psilemma} implies that the $\alpha$-mixing coefficients of the sequence $P_n$ satisfy $\alpha (\ell) \ll e^{-\tau \ell}$, thus
\[ |\mathbb{E} (f(P_n) f(P_{n+\ell})) -\mathbb{E} (f(P_n)) \mathbb{E} (f(P_{n+\ell}))| \ll e^{-\tau \ell} . \]
Here $|\mathbb{E} (f(P_{n+\ell}))| \ll e^{-\tau (n+\ell)}$ and $|\mathbb{E} (f(P_n))| \ll 1$. Hence $|\mathbb{E} (f(P_n) f(P_{n+\ell}))| \ll e^{-\tau \ell}$, and
\[ \left| \sum_{\ell =R}^{N-1} \sum_{n=M+1}^{M+N-\ell} \mathbb{E} (f(P_n) f(P_{n+\ell})) \right| \ll \sum_{\ell=R}^{N-1} N e^{-\tau \ell} \ll N e^{-\tau R} . \]
Now let $1 \le \ell \le R$, and consider
\[ \mathbb{E} (f(P_n) f(P_{n+\ell})) = \sum_{\substack{g \in G_{(-1)^n} \\ h \in G_{(-1)^{\ell}}}} f(g) f(gh) \mu \left( \{ P_n=g \} \cap \{ P_{n+1,n+\ell} =h \} \right) . \]
Theorem \ref{gausskuzmintheorem} shows that here $\mu( \{ P_n=g \} \cap \{ P_{n+1,n+\ell} =h \} ) = \mu_{\mathrm{Gauss}} ( P_\ell =h) / |G_1| +O(e^{-\tau n})$. Therefore
\[ \mathbb{E} (f(P_n) f(P_{n+\ell})) = \sum_{\substack{g \in G_{(-1)^n} \\ h \in G_{(-1)^{\ell}}}} f(g) f(gh) \frac{\mu_{\mathrm{Gauss}}(P_{\ell} =h)}{|G_1|} + O(e^{-\tau n}) = \mathbb{E} (f(U_{(-1)^n}) f(U_{(-1)^n} P_{\ell})) + O(e^{-\tau n}) , \]
and
\[ \begin{split} \sum_{\ell =1}^R \sum_{n=M+1}^{M+N-\ell} \mathbb{E} (f(P_n) f(P_{n+\ell})) &= \sum_{\ell =1}^R \frac{N-\ell}{2} \left( \mathbb{E} (f(U_1) f(U_1 P_{\ell})) + \mathbb{E} (f(U_{-1}) f(U_{-1} P_{\ell})) \right) + O(R) \\ &= \frac{N}{2} \sum_{\ell =1}^{\infty} \left( \mathbb{E} (f(U_1) f(U_1 P_{\ell})) + \mathbb{E} (f(U_{-1}) f(U_{-1} P_{\ell})) \right) +O(R +N e^{-\tau R}). \end{split} \]
The previous estimates for the right-hand side of \eqref{expandsquare} and the definition \eqref{sigmadef} of $\sigma_f$ lead to
\[ \mathbb{E} \left( \sum_{n=M+1}^{M+N} f(P_n) \right)^2 = \sigma_f^2 N +O(R +N e^{-\tau R}) , \]
and the claim follows by choosing $R \approx \log (N+1)$.
\end{proof}

\begin{proof}[Proof of Theorem \ref{Gtheorem}] Claim (i) was proved in Lemma \ref{variancelemma}. Now let $\alpha \sim \mu$ with $\mu \ll \lambda$, and assume \eqref{lipschitz}. Lemmas \ref{psilemma} and \ref{variancelemma} show that the sequence of random variables $f(P_n)$ is $\alpha$-mixing with exponential rate, and satisfy $\mathbb{E} ( \sum_{n=M+1}^{M+N} f(P_n) - E_f N)^2 = \sigma_f^2 N+O(\log (N+1))$ uniformly in $M \ge 0$. Assuming $\sigma_f>0$, a general result of Philipp and Stout \cite[Theorem 7.1]{PS} on partial sums of nonstationary $\alpha$-mixing random variables shows that without changing its distribution, the process $\sum_{1 \le n \le t} f(P_n)$ can be redefined on a richer probability space so that $\sum_{1 \le n \le t} f(P_n) - E_f t = \sigma_f W(t) +O(t^{1/2-\eta})$ a.s.\ with a universal constant $\eta>0$. This proves (iii) in the case $\sigma_f >0$.

As Philipp and Stout assume $\sigma_f>0$ throughout their treatise, for the sake of completeness we include a proof of (iii) in the case $\sigma_f=0$, i.e.\ when $\mathbb{E} ( \sum_{n=M+1}^{M+N} f(P_n) - E_f N)^2 \ll \log (N+1)$ uniformly in $M \ge 0$. Fix a small $\varepsilon>0$. An application of the Chebyshev inequality gives
\[ \mu \left( \left| \sum_{n=1}^{N^2} f(P_n) - E_f N^2 \right| \ge N^{1/2+\varepsilon} \right) \ll \frac{\log (N+1)}{N^{1+2\varepsilon}} . \]
Since $f(P_n)$ is bounded and $\alpha$-mixing with exponential rate, for all $p>2$ we have $\mathbb{E} (|\sum_{n=M+1}^{M+N} f(P_n) - E_f N|^p) \ll N^{p/2}$ uniformly in $M \ge 0$ and $N \ge 1$ \cite{RI}. The Erd\H{o}s--Stechkin inequality \cite{MOR} strengthens this to $\mathbb{E} (\max_{1 \le k \le N} |\sum_{n=M+1}^{M+k} f(P_n) - E_f k|^p) \ll N^{p/2}$. Choosing a suitably large $p>2$ thus leads to
\[ \mu \left( \max_{1 \le k \le 2N} \left| \sum_{n=N^2 +1}^{N^2+k} f(P_n) - E_f k \right| \ge N^{1/2+\varepsilon} \right) \ll \frac{N^{p/2}}{N^{p(1/2+\varepsilon)}} \ll \frac{1}{N^2} . \]
An application of the Borel--Cantelli lemma then shows that
\[  \left| \sum_{n=1}^{N^2} f(P_n) - E_f N^2 \right| \ll N^{1/2+\varepsilon} \quad \textrm{and} \quad  \max_{1 \le k \le 2N} \left| \sum_{n=N^2 +1}^{N^2+k} f(P_n) - E_f k \right| \ll N^{1/2+\varepsilon} \quad \textrm{for a.e. } \alpha . \]
In particular, $|\sum_{n=1}^N f(P_n) - E_f N| \ll N^{1/4+\varepsilon}$ for a.e.\ $\alpha$. This proves (iii) in the case $\sigma_f=0$.

The almost sure approximation by a Wiener process in part (iii) immediately implies the functional CLT under the assumption \eqref{lipschitz}. See Peligrad \cite{PE} for a direct proof of the functional CLT under even weaker mixing assumptions. Exactly as in the proof of Corollary \ref{gausskuzmincorollary}, we can easily remove assumption \eqref{lipschitz} on the density from the functional CLT. This proves (ii).
\end{proof}

\section{Limit laws for the local discrepancy}\label{limitlawsection}

We rely on an explicit formula of Ro\c cadas and Schoissengeier \cite{SCH1}, who showed that for any irrational $\alpha \in [0,1]$, any $r \in (0,1)$ and any $k \ge 0$,
\begin{equation}\label{schoissengeier}
\begin{split} \max_{0 \le N < q_{k+1}} S_{N,r} (\alpha) &= \sum_{\substack{j=0 \\ j \textrm{ even}}}^k \left\{ q_j r \right\} \left( \left( 1-\{ q_j r \} \right) a_{j+1} + \{ q_{j+1} r \} - \{ q_{j-1} r \} \right) +O(1), \\ \min_{0 \le N < q_{k+1}} S_{N,r} (\alpha) &= -\sum_{\substack{j=0 \\ j \textrm{ odd}}}^k \left\{ q_j r \right\} \left( \left( 1-\{ q_j r \} \right) a_{j+1} + \{ q_{j+1} r \} - \{ q_{j-1} r \} \right) +O(1) \end{split}
\end{equation}
with universal implied constants. We give the proof of Theorem \ref{discrepancytheorem} after a preparatory lemma. Let $\theta_m$ be as in \eqref{thetam}, and recall that $\gamma$ denotes the Euler--Mascheroni constant.
\begin{lem}\label{Xjlemma} Let $\alpha \sim \mu$ with $\mu \ll \lambda$, and assume that $\frac{\mathrm{d}\mu}{\mathrm{d}\lambda}$ is positive and Lipschitz. Let $l/m \in (0,1)$ be a reduced fraction, and let
\[ X_j = \left\{ q_{j-1} \frac{l}{m} \right\} \left( 1- \left\{ q_{j-1} \frac{l}{m} \right\} \right) a_j . \]
Then
\[ \left( \frac{\sum_{j=1}^k X_{2j-1} - A_k}{s_k}, \frac{\sum_{j=1}^k X_{2j} - A_k}{s_k} \right) \overset{d}{\to} \mathrm{Stab}(1,1) \otimes \mathrm{Stab} (1,1) , \]
where $A_k = \frac{1}{6 \log 2} k \log k - \frac{1}{6 \log 2} \left( 6 \theta_m + \gamma+\log \frac{12 \log 2}{\pi} \right) k$ and $s_k = \frac{\pi}{12 \log 2} k$.
\end{lem}

\begin{proof} Fix real numbers $x_1, x_2$ such that $(x_1, x_2) \neq (0, 0)$, and let $Y_j=(x_1/k) X_{2j-1} + (x_2/k) X_{2j}$, $1 \le j \le k$. Throughout, implied constants are allowed to depend on $x_1, x_2$. The random variable $Y_j$ is a function of $(a_{2j-1}, a_{2j}, P_{2j-1}, P_{2j})$, therefore by Lemma \ref{psilemma} the sequence $Y_j$ is $\psi$-mixing with exponential rate.

Heinrich \cite[Lemma 1]{HE} proved the following result for a sequence of random variables $Y_1, Y_2, \ldots$ with $\psi$-mixing coefficients $\psi (\ell)$, $\ell \ge 1$. Let
\[ \sigma_k = \max_{1 \le j \le k} \mathbb{E} \left( \left| e^{i Y_j} -1 \right| \right) \qquad  \textrm{and} \qquad \Sigma_k = \sum_{j=1}^k \mathbb{E} \left( \left| e^{i Y_j} -1 \right| \right) . \]
Assuming
\begin{equation}\label{heinrichconditions}
\begin{split} \sigma_k &\le \min \left\{ \frac{1}{2 (1+\psi (1))^2 (2m_0+1)^2}, \frac{1}{2 (1+\psi (1)) (2 p_0 m_0 +1)} \right\} , \\ \sigma_k \Sigma_k &\le \frac{1}{2} \left( 9+\sum_{\ell=1}^{m_0} \psi (\ell) \right)^{-1} \sum_{j=1}^k \mathbb{E} \left( 1-\cos Y_j \right) \end{split}
\end{equation}
with some integers $p_0 \ge 2$ and $m_0 \ge 1$, we have
\begin{equation}\label{heinrichclaim}
\begin{split} \Bigg| \mathbb{E} \Bigg( &\exp \Bigg( i \sum_{j=1}^k Y_j \Bigg) \Bigg) - \exp \Bigg( \sum_{j=1}^k \mathbb{E} \left( e^{i Y_j} -1 \right) \Bigg) \Bigg| \\ &\le \left( 9+\sum_{\ell=1}^{m_0} \psi (\ell) \right) \sigma_k \Sigma_k \exp \Bigg( - \frac{1}{2} \sum_{j=1}^k \mathbb{E} (1-\cos Y_j) \Bigg) + \left( 2^{-p_0} + (6+\psi (1)) \psi (m_0) \right) \Sigma_k . \end{split}
\end{equation}

One readily checks that if $x_1 \neq 0$, then with a suitable constant $B>0$,
\[ \begin{split} \mathbb{E} (1-\cos Y_j) &\ge \mu \left( \frac{\pi}{2} \le |Y_j| \le \frac{3\pi}{2} \right) \\ &\gg \mu \left( \{ q_{2j-2} \equiv 1 \pmod{m} \} \cap \left\{ \frac{B}{2}k \le a_{2j-1} \le Bk \right\} \cap \{ a_{2j}=1 \} \right) \\ &\gg \mu_{\mathrm{Gauss}} \left( \frac{B}{2} k \le a_{2j-1} \le Bk \right) \gg \frac{1}{k} . \end{split} \]
A similar argument shows that $\mathbb{E}(1-\cos Y_j) \gg 1/k$ holds in the case $x_1=0$, $x_2 \neq 0$ as well. Further, we have
\[ \mathbb{E} \left( |e^{iY_j}-1| \right) \le \mathbb{E} \left( \min \{ |Y_j|, 2 \} \right) \ll \mathbb{E} \left( \min \left\{ \frac{a_{2j-1}+a_{2j}}{k}, 1 \right\} \right) \ll \frac{\log k}{k} , \]
which shows that $\sigma_k \ll (\log k)/k$ and $\Sigma_k \ll \log k$. Using the fact that $\sum_{\ell=1}^{\infty} \psi (\ell) < \infty$, we see that conditions \eqref{heinrichconditions} are satisfied with $p_0=m_0 \approx \sqrt{k/\log k}$. The estimate \eqref{heinrichclaim} thus yields
\begin{equation}\label{heinrich}
\mathbb{E} \Bigg( \exp \Bigg( i \sum_{j=1}^k Y_j \Bigg) \Bigg) = \exp \Bigg( \sum_{j=1}^k \mathbb{E} (e^{i Y_j}-1) \Bigg) + O \left( \frac{(\log k)^2}{k} \right) .
\end{equation}
Here $\sum_{j \ll \log k} \mathbb{E} (e^{i Y_j}-1) =O((\log k)^2/k)$, hence it will be enough to consider the terms $j \gg \log k$.

We can express $Y_j$ as $Y_j=F(P_{2j-2}, a_{2j-1}, a_{2j})$ with the function
\[ F \left( \left( \begin{array}{ll} a & b \\ c & d \end{array} \right), b_1, b_2 \right) = \frac{x_1}{k} \left\{ d \frac{l}{m} \right\} \left( 1-\left\{ d \frac{l}{m} \right\} \right) b_1 + \frac{x_2}{k} \left\{ (b_1 d+c) \frac{l}{m} \right\} \left( 1-\left\{ (b_1 d+c) \frac{l}{m} \right\} \right) b_2 . \]
Theorem \ref{gausskuzmintheorem} yields
\[ \begin{split} \mathbb{E} (e^{i Y_j}-1) &= \sum_{\substack{g \in G_1 \\ b_1, b_2 \in \mathbb{N}}} (e^{i F(g, b_1, b_2)} -1) \mu \left( P_{2j-2}=g, a_{2j-1}=b_1, a_{2j}=b_2 \right) \\ &= \sum_{\substack{g \in G_1 \\ b_1, b_2 \in \mathbb{N}}} (e^{i F(g, b_1, b_2)} -1) \frac{\mu_{\mathrm{Gauss}} (a_1=b_1, a_2=b_2)}{|G_1|} + O (e^{-\tau j}) . \end{split} \]
For a fixed $g=\left( \begin{array}{cc} a & b \\ c & d \end{array} \right) \in G_1$, we have $F(g, b_1, b_2) = t_1 b_1 + t_2 (b_1) b_2$ with
\[ t_1 = \frac{x_1}{k} \left\{ d \frac{l}{m} \right\} \left( 1-\left\{ d \frac{l}{m} \right\} \right), \qquad t_2 (b_1) = \frac{x_2}{k} \left\{ (b_1 d+c) \frac{l}{m} \right\} \left( 1-\left\{ (b_1 d+c) \frac{l}{m} \right\} \right) . \]
In \cite[Lemma 20]{BO} it was shown that for any $t_1, t_2 \in (-1/2, 1/2)$, we have
\[ \begin{split} \sum_{b_1, b_2 \in \mathbb{N}} \left( e^{i(t_1 b_1 + t_2 b_2)}-1 \right) &\mu_{\mathrm{Gauss}} (a_1 = b_1, a_2 = b_2) \\ = &-\frac{1}{\log 2} \left( i \gamma t_1 + \frac{\pi}{2} |t_1| + i t_1 \log |t_1|  + i \gamma t_2 + \frac{\pi}{2} |t_2| + i t_2 \log |t_2| \right) \\ &+O \left( t_1^2 \log \frac{1}{|t_1|} + t_2^2 \log \frac{1}{|t_2|} + |t_1 t_2| \log \frac{1}{|t_1|} \log \frac{1}{|t_2|} \right) . \end{split} \]
The proof actually gives that more generally, for any constant $t_1$ and any sequence $t_2(n)$, $n \in \mathbb{N}$ with $t_1, t_2(n) \in (-1/2, 1/2)$, we have
\[ \begin{split} &\sum_{b_1, b_2 \in \mathbb{N}} \left( e^{i(t_1 b_1 + t_2 (b_1) b_2)}-1 \right) \mu_{\mathrm{Gauss}} (a_1 = b_1, a_2 = b_2) \\ = &-\frac{1}{\log 2} \left( i \gamma t_1 + \frac{\pi}{2} |t_1| + i t_1 \log |t_1|  - i \sum_{b_1, b_2 \in \mathbb{N}} t_2(b_1) b_2 R(b_1, b_2) + \sum_{b_1 \in \mathbb{N}} \frac{\frac{\pi}{2} |t_2 (b_1)| + it_2 (b_1) \log |t_2 (b_1)|}{b_1 (b_1+1)} \right) \\ &+O \left( t_1^2 \log \frac{1}{|t_1|} + T_2^2 \log \frac{1}{T_2} + |t_1| T_2 \log \frac{1}{|t_1|} \log \frac{1}{T_2} \right) , \end{split} \]
where $T_2 = \sup_{b_1 \in \mathbb{N}} |t_2 (b_1)|$, and
\[ R(b_1, b_2) = \mu_{\mathrm{Gauss}} (a_1 = b_1, a_2 = b_2) - \frac{1}{b_1 (b_1+1) b_2 (b_2 +2)} . \]
As we observed in \cite[Lemma 20]{BO} using telescoping sums, $\sum_{b_1, b_2 \in \mathbb{N}} b_2 R(b_1, b_2) = -\gamma$. In particular, we obtain a formula for $\mathbb{E} (e^{i Y_j}-1)$ in the form of an average over $G_1$. Lemma \ref{SL2Zmlemma} and formula \eqref{sumovera} show that
\[ \begin{split} \frac{1}{|G_1|} \sum_{g \in G_1} \left\{ d \frac{l}{m} \right\} \left( 1 - \left\{ d \frac{l}{m} \right\} \right) = \sum_{a \in \mathbb{Z}_m} \nu_a \left\{ \frac{a}{m} \right\} \left( 1-\left\{ \frac{a}{m} \right\} \right) &=\frac{1}{6} , \\ \frac{1}{|G_1|} \sum_{g \in G_1} \left\{ d \frac{l}{m} \right\} \left( 1 - \left\{ d \frac{l}{m} \right\} \right) \log \left( \left\{ d \frac{l}{m} \right\} \left( 1 - \left\{ d \frac{l}{m} \right\} \right) \right) &= \theta_m. \end{split} \]
Note that the previous two formulas do not depend on $l$ since $\nu_a$ is invariant under multiplication by elements of $\mathbb{Z}_m^*$. We saw in the proof of Lemma \ref{SL2Zmlemma} that $G_1$ acts transitively on the set of row vectors $V$, and that transposition is a bijection of $G_1$. Consequently for any fixed $b_1 \in \mathbb{N}$, the row vector $(1, b_1 \pmod{m}) \left( \begin{array}{cc} a & b \\ c & d \end{array} \right)^{\top}$ is uniformly distributed on $V$, and its second coordinate, $b_1 d +c \pmod{m}$ has distribution $\nu$. In particular, for any fixed $b_1 \in \mathbb{N}$ we have the same averages
\[ \begin{split} \frac{1}{|G_1|} \sum_{g \in G_1} \left\{ (b_1 d +c) \frac{l}{m} \right\} \left( 1 - \left\{ (b_1 d+c) \frac{l}{m} \right\} \right) = \sum_{a \in \mathbb{Z}_m} \nu_a \left\{ \frac{a}{m} \right\} \left( 1-\left\{ \frac{a}{m} \right\} \right) &=\frac{1}{6} , \\ \frac{1}{|G_1|} \sum_{g \in G_1} \left\{ (b_1 d +c) \frac{l}{m} \right\} \left( 1 - \left\{ (b_1 d +c) \frac{l}{m} \right\} \right) \log \left( \left\{ (b_1 d+c) \frac{l}{m} \right\} \left( 1 - \left\{ (b_1 d +c) \frac{l}{m} \right\} \right) \right) &= \theta_m. \end{split} \]
After some simplification, we arrive at
\[ \begin{split} \mathbb{E} (e^{i Y_j}-1) = &i \frac{\log k -6\theta_m -\gamma}{6\log 2} \cdot \frac{x_1}{k} - \frac{\pi}{(12 \log 2)k} \left( |x_1| + \frac{2i}{\pi} x_1 \log |x_1| \right) \\ &+i \frac{\log k -6\theta_m -\gamma}{6\log 2} \cdot \frac{x_2}{k} - \frac{\pi}{(12 \log 2)k} \left( |x_2| + \frac{2i}{\pi} x_2 \log |x_2| \right) + O \left( \frac{(\log k)^2}{k^2} + e^{-\tau j} \right) . \end{split} \]
We can compute the right-hand side of \eqref{heinrich} by summing over $\log k \ll j \le k$, thus
\[ \begin{split} \mathbb{E} \Bigg( \exp \Bigg( i \sum_{j=1}^k Y_j \Bigg) \Bigg) = &\exp \bigg( i \frac{\log k -6\theta_m -\gamma}{6\log 2} x_1 - \frac{\pi}{12 \log 2} \left( |x_1| + \frac{2i}{\pi} x_1 \log |x_1| \right) \\ &\hspace{10mm}+i \frac{\log k -6\theta_m -\gamma}{6\log 2} x_2 - \frac{\pi}{12 \log 2} \left( |x_2| + \frac{2i}{\pi} x_2 \log |x_2| \right) + O \left( \frac{(\log k)^2}{k} \right) \bigg) \\ &+ O \left( \frac{(\log k)^2}{k} \right) .  \end{split} \]
Letting $A_k$ and $s_k$ be as in the claim and replacing $x_n$ by $\frac{12 \log 2}{\pi} x_n$, $n=1,2$, we obtain
\[ \begin{split} \mathbb{E} \bigg( \exp \bigg( i &\frac{\sum_{j=1}^k X_{2j-1} - A_k}{s_k} x_1 + i \frac{\sum_{j=1}^k X_{2j} - A_k}{s_k} x_2 \bigg) \bigg) \\ &= \exp \left( - \left( |x_1| + \frac{2i}{\pi} x_1 \log |x_1| \right) \right) \exp \left( - \left( |x_2| + \frac{2i}{\pi} x_2 \log |x_2| \right) \right) +O \left( \frac{(\log k)^2}{k} \right) . \end{split} \]
This implies the pointwise convergence of the characteristic function of the random vector in the claim to that of $\mathrm{Stab}(1,1) \otimes \mathrm{Stab}(1,1)$, which proves the desired convergence in distribution.
\end{proof}

\begin{proof}[Proof of Theorem \ref{discrepancytheorem}] Fix a reduced fraction $r=l/m \in (0,1)$. We may assume that the density $\frac{\mathrm{d}\mu}{\mathrm{d}\lambda}$ is positive and Lipschitz. This assumption can be removed exactly as in the proof of Corollary \ref{gausskuzmincorollary}.

For any $M \ge 1$, let $k_M^*=k_M^*(\alpha)$ denote the random index for which $q_{k_M^*} \le M < q_{k_M^*+1}$, and let $k_M$ be the odd integer closest to $\frac{12 \log 2}{\pi^2} \log M$. Using the fact that $\log q_k$ satisfies the CLT with centering term $\frac{\pi^2}{12 \log 2} k$ and scaling term $k^{1/2}$ \cite[p.\ 194]{IK}, we immediately obtain $\mu (|\log q_{k_M} - \log M| \ge (\log M)^{1/2+\varepsilon} ) \to 0$ with any $\varepsilon>0$. Consequently, $\mu (|k_M^*-k_M| \ge (\log M)^{1/2+\varepsilon}) \to 0$. The explicit formula \eqref{schoissengeier} thus yields
\begin{equation}\label{maxminexplicit}
\begin{split} \max_{0 \le N<M} S_{N,r}(\alpha) &= \sum_{j=0}^{\frac{k_M-1}{2}} \{ q_{2j} r \} \left( (1-\{ q_{2j} r \}) a_{2j+1} + \{ q_{2j+1}r \} - \{ q_{2j-1} r \} \right) + \xi_M(\alpha), \\ \min_{0 \le N<M} S_{N,r}(\alpha) &= - \sum_{j=0}^{\frac{k_M-1}{2}} \{ q_{2j+1} r \} \left( (1-\{ q_{2j+1} r \}) a_{2j+2} + \{ q_{2j+2}r \} - \{ q_{2j} r \} \right) + \xi_M'(\alpha) \end{split}
\end{equation}
with error terms $\xi_M(\alpha)$, $\xi_M'(\alpha)$ which, outside a set of $\mu$-measure $o(1)$, satisfy
\[ | \xi_M(\alpha)|, |\xi_M'(\alpha)| \le \sum_{j=k_M-(\log M)^{1/2+\varepsilon}}^{k_M+(\log M)^{1/2+\varepsilon}} a_j +O(1) . \]
A classical result of Khintchine \cite[p.\ 204]{IK} states that $\sum_{j=1}^k a_j / (k \log k) \to 1/\log 2$ in measure. Hence
\[ \mu \left( \sum_{j=k_M-(\log M)^{1/2+\varepsilon}}^{k_M+(\log M)^{1/2+\varepsilon}} a_j \ge (\log M)^{1/2+2 \varepsilon} \right) \ll \mu_{\mathrm{Gauss}} \left( \sum_{1 \le j \ll (\log M)^{1/2+\varepsilon}} a_j \ge (\log M)^{1/2+2\varepsilon} \right) \to 0. \]
In particular, $\xi_M(\alpha), \xi_M'(\alpha) = o(\log M)$ in $\mu$-measure, and are thus negligible.

Relation \eqref{Pnae} shows that
\[ \sum_{j=0}^{\frac{k_M-1}{2}} \{ q_{2j} r \} \cdot \{ q_{2j-1} r \} = \rho k_M +o(k_M) \quad \textrm{and} \quad \sum_{j=0}^{\frac{k_M-1}{2}} \{ q_{2j} r \} \cdot \{ q_{2j+1} r \} = \rho k_M +o(k_M) \]
hold for a.e.\ $\alpha$ with the same constant $\rho$. In particular, the error terms in the previous formula are $o(\log M)$ in $\mu$-measure, and \eqref{maxminexplicit} simplifies to
\[ \begin{split} \max_{0 \le N<M} S_{N,r}(\alpha) &= \sum_{j=0}^{\frac{k_M-1}{2}} \{ q_{2j} r \} (1-\{ q_{2j} r \}) a_{2j+1} + o(\log M) \quad \textrm{in $\mu$-measure}, \\ \min_{0 \le N<M} S_{N,r}(\alpha) &= - \sum_{j=0}^{\frac{k_M-1}{2}} \{ q_{2j+1} r \} (1-\{ q_{2j+1} r \}) a_{2j+2} +o(\log M) \quad \textrm{in $\mu$-measure} . \end{split} \]
Lemma \ref{Xjlemma} yields the desired limit law with centering term
\begin{multline*}
\frac{1}{6 \log 2} \cdot \frac{k_M}{2} \log \frac{k_M}{2} - \frac{6 \theta_m +\gamma +\log \frac{12 \log 2}{\pi}}{6 \log 2} \cdot \frac{k_M}{2} \\ = \frac{1}{\pi^2} \log M \log \log M - \frac{6 \theta_m + \gamma + \log (2 \pi)}{\pi^2} \log M +O(\log \log M)
\end{multline*}
and scaling term $\frac{\pi}{12 \log 2} \cdot \frac{k_M}{2} = \frac{1}{2 \pi} \log M +O(1)$.
\end{proof}

\section*{Acknowledgments}

The author is supported by the Austrian Science Fund (FWF) project M 3260-N. I would like to thank the referee for a careful reading of the manuscript and valuable comments.


\begin{thebibliography}{99}
\footnotesize{

\bibitem{AK} J.\ Aaronson and M.\ Keane: \textit{The visits to zero of some deterministic random walks.} Proc.\ London Math.\ Soc.\ 44 (1982), 535--553.

\bibitem{AG1} G.\ Aggarwal and A.\ Ghosh: \textit{A generalized L\'evy--Khintchine theorem.} arXiv:2408.15683.

\bibitem{AG2} G.\ Aggarwal and A.\ Ghosh: \textit{Counting and joint equidistribution of approximates.} arXiv:2401.02747.

\bibitem{AG3} G.\ Aggarwal and A.\ Ghosh: \textit{Two Central limit theorems in Diophantine approximation.} arXiv:2306.02304.

\bibitem{AGY} M.\ Alam, A.\ Ghosh and S.\ Yu: \textit{Quantitative Diophantine approximation with congruence conditions.} J.\ Th\'eor.\ Nombres Bordeaux 33 (2021), 261--271.

\bibitem{ADDS} A.\ Avila, D.\ Dolgopyat, E.\ Duryev and O.\ Sarig: \textit{The visits to zero of a random walk driven by an irrational rotation.} Israel J.\ Math.\ 207 (2015), 653--717.

\bibitem{BE} J.\ Beck: \textit{Probabilistic Diophantine Approximation. Randomness in Lattice Point Counting.} Springer Monographs in Mathematics. Springer, Cham, 2014.

\bibitem{BI} P.\ Billingsley: \textit{Convergence of Probability Measures.} Second Edition. John Wiley \& Sons, Inc., New York, 1999.

\bibitem{BO} B.\ Borda: \textit{Limit laws of maximal Birkhoff sums for circle rotations via quantum modular forms.} Int.\ Math.\ Res.\ Not.\ IMRN 2023 (2023), 19340--19389.

\bibitem{BR} R.\ Bradley: \textit{Basic properties of strong mixing conditions. A survey and some open questions.} Update of, and a supplement to, the 1986 original. Probab.\ Surv.\ 2 (2005), 107--144.

\bibitem{BU} M.\ Bromberg and C.\ Ulcigrai: \textit{A temporal central limit theorem for real-valued cocycles over rotations.} Ann.\ Inst.\ Henri Poincar\'e Probab.\ Stat.\ 54 (2018), 2304--2334.

\bibitem{DK} K.\ Dajani and C.\ Kraaikamp: \textit{A note on the approximation by continued fractions under an extra condition.} New York J.\ Math.\ 3A (1997/98), 69--80.

\bibitem{DF1} D.\ Dolgopyat and B.\ Fayad: \textit{Deviations of ergodic sums for toral translations I. Convex bodies.} Geom.\ Funct.\ Anal.\ 24 (2014), 85--115.

\bibitem{DF2} D.\ Dolgopyat and B.\ Fayad: \textit{Deviations of ergodic sums for toral translations II. Boxes.} Publ.\ Math.\ Inst.\ Hautes \'Etudes Sci.\ 132 (2020), 293--352.

\bibitem{DF3} D.\ Dolgopyat and B.\ Fayad: \textit{Limit theorems for toral translations.} Hyperbolic dynamics, fluctuations and large deviations, 227--277, Proc.\ Sympos.\ Pure Math., 89, Amer.\ Math.\ Soc., Providence, RI, 2015.

\bibitem{FS} A.\ Fisher and T.\ Schmidt: \textit{Distribution of approximants and geodesic flows.} Ergodic Theory Dynam.\ Systems 34 (2014), 1832--1848.

\bibitem{HE} L.\ Heinrich: \textit{Rates of convergence in stable limit theorems for sums of exponentially $\psi$-mixing random variables with an application to metric theory of continued fractions.} Math.\ Nachr.\ 131 (1987), 149--165.

\bibitem{IK} M.\ Iosifescu and C.\ Kraaikamp: \textit{Metrical Theory of Continued Fractions.} Mathematics and its Applications, 547. Kluwer Academic Publishers, Dordrecht, 2002.

\bibitem{JL} H.\ Jager and P.\ Liardet: \textit{Distributions arithm\'etiques des d\'enominateurs de convergents de fractions continues.} Nederl.\ Akad.\ Wetensch.\ Indag.\ Math.\ 50 (1988), 181--197.

\bibitem{KE3} H.\ Kesten: \textit{The discrepancy of random sequences $\{ kx \}$.} Acta Arith.\ 10 (1964/65), 183--213.

\bibitem{KE1} H.\ Kesten: \textit{Uniform distribution mod 1.} Ann.\ of Math.\ (2) 71 (1960), 445--471.

\bibitem{KE2} H.\ Kesten: \textit{Uniform distribution mod 1. II.} Acta Arith.\ 7 (1961/1962), 355--380.

\bibitem{LI} C.\ Liverani: \textit{Decay of correlations.} Ann.\ of Math.\ (2) 142 (1995), 239--301.

\bibitem{MO} R.\ Moeckel: \textit{Geodesics on modular surfaces and continued fractions.} Ergodic Theory Dynam.\ Systems 2 (1982), 69--83.

\bibitem{MOR} F.\ M\'oricz: \textit{Moment inequalities and the strong law of large numbers.} Z.\ Wahrscheinlichkeitstheorie und Verw.\ Gebiete 35 (1976), 299--314.

\bibitem{NRS} E.\ Nesharim, R.\ R\"uhr and R.\ Shi: \textit{Metric Diophantine approximation with congruence conditions.} Int.\ J.\ Number Theory 16 (2020), 1923--1933.

\bibitem{PE} M.\ Peligrad: \textit{Invariance principles for mixing sequences of random variables.} Ann.\ Probab.\ 10 (1982), 968--981.

\bibitem{PS} W.\ Philipp and W.\ Stout: \textit{Almost sure invariance principles for partial sums of weakly dependent random variables.} Mem.\ Amer.\ Math.\ Soc.\ 2 (1975), no.\ 161.

\bibitem{RI} E.\ Rio: \textit{Asymptotic Theory of Weakly Dependent Random Processes.} Translated from the 2000 French edition. Probability Theory and Stochastic Modeling, 80. Springer, Berlin, 2017.

\bibitem{SCH1} L.\ Ro\c cadas and J.\ Schoi\ss engeier: \textit{On the local discrepancy of $(n\alpha)$-sequences.} J.\ Number Theory 131 (2011), 1492--1497.

\bibitem{SCH2} J.\ Schoissengeier: \textit{A metrical result on the discrepancy of $(n \alpha)$.} Glasgow Math.\ J.\ 40 (1998), 393--425.

\bibitem{SW} U.\ Shapira and B.\ Weiss: \textit{Geometric and arithmetic aspects of approximation vectors.} arXiv:2206.05329.

\bibitem{SZ} P.\ Sz\"usz: \textit{Verallgemeinerung und Anwendungen eines Kusminschen Satzes.} Acta Arith.\ 7 (1961/62), 149--160.

}
\end{thebibliography}
\end{document}